\documentclass[a4paper]{article}
\usepackage{a4wide}
\usepackage[utf8]{inputenc}
\usepackage[ngerman, english]{babel}
\usepackage{amsmath,amsthm,amssymb,amsfonts}
\usepackage{mathrsfs}
\usepackage{ifthen}
\usepackage{enumitem}
\usepackage{comment}
\usepackage{subcaption}
\usepackage{hyphsubst}
\usepackage{xcolor}
\usepackage{calc}
\usepackage{graphicx}
\usepackage{pgfplots}
\pgfplotsset{compat=1.5.1}
\usepackage{hyperref}
\usepackage[toc,page]{appendix}
\usepackage{cleveref}
\usepackage{leftindex}
\usepackage{algorithm}
\usepackage{algpseudocode}
\usepackage{booktabs}
\usepackage[bbgreekl]{mathbbol}
\DeclareSymbolFontAlphabet{\mathbbm}{bbold}
\DeclareSymbolFontAlphabet{\mathbb}{AMSb}
\usepackage{bbold}

\newcommand{\jv}[1]{J_{#1}(2|g|)}
\newcommand{\jc}[1]{(\calF c)_{#1}}

\makeatletter
\def\@hspace#1{\begingroup\setlength\dimen@{#1}\hskip\dimen@\endgroup}
\makeatother

\newtheorem{theorem}{Theorem}[section]
\newtheorem{lemma}[theorem]{Lemma}
\newtheorem{definition}[theorem]{Definition}
\newtheorem{remark}[theorem]{Remark}
\newtheorem{corollary}[theorem]{Corollary}

\def\yobs{y^{\text{obs}}}

\def\BopII{\mathcal{B}_{2}}

\def\BopIpsdtr{\mathcal{D}}

\def\cycconv{*_{\text{cyc}}}
\def\Ucont{U_{\theta}}
\newcommand{\trunc}[1]{#1}

\newcommand{\dd}{\mathrm{d}}

\newcommand{\bpm}{\begin{pmatrix}}
\newcommand{\epm}{\end{pmatrix}}

\DeclareMathOperator{\tr}{tr}

\DeclareMathOperator{\argmin}{argmin}

\DeclareMathOperator{\Herm}{Herm}

\newcommand{\setC}{\mathbb{C}}

\newcommand{\setN}{\mathbb{N}}

\newcommand{\setR}{\mathbb{R}}

\newcommand{\setZ}{\mathbb{Z}}

\newcommand{\calB}{\mathcal{B}}

\newcommand{\calF}{\mathcal{F}}
\newcommand{\calG}{\mathcal{G}}

\newcommand{\calI}{\mathcal{I}}
\newcommand{\calJ}{\mathcal{J}}

\newcommand{\calO}{\mathcal{O}}

\newcommand{\calT}{\mathcal{T}}

\definecolor{brickred}{rgb}{0.8, 0.25, 0.33}
\definecolor{bostonuniversityred}{rgb}{0.8, 0.0, 0.0}
\definecolor{cornellred}{rgb}{0.7, 0.11, 0.11}
\definecolor{corn}{rgb}{0.98, 0.93, 0.36}
\definecolor{schoolbusyellow}{rgb}{1.0, 0.85, 0.0}
\definecolor{TUblue}{rgb}{0,102,153}
\colorlet{TUbluelight}{TUblue!30!white}
\usepackage{fancyhdr}
\usepackage[bbgreekl]{mathbbol}
\usepackage{authblk}

\author[1]{Florian Oberender}
\affil[1]{Institut f\"ur Numerische und Angewandte Mathematik, Georg-August Universität Göttingen}
\title{Well-posedness and instability of free electron quantum tomography%
\footnote{The author acknowledges support from DFG, CRC 1456 project 432680300.}}

\begin{document}

\maketitle
\begin{abstract}
\noindent
Recent advancements in photon induced near-field electron microscopy (PINEM) enable the preparation, coherent manipulation and characterization of free-electron quantum states. The available measurement consists of electron energy spectrograms and the goal is the reconstruction of a density matrix representing the quantum state. This requires the solution of a constrained linear inverse problem, where a positive semi-definite trace-class operator is reconstructed given its diagonal in different bases. We show the well-posedness of this problem by exploiting the regularizing effect of the positive semi-definiteness constraint. Unusually, well-posedness in this case does not imply any stability estimates. We show that no global stability estimates exist and any estimator converges arbitrarily slowly. 
We also provide further bounds on the instability generally complementing the analysis done in \cite{Shi:20}. Furthermore, we derive a decomposition of the discretized operator which allows us to study its injectivity and stability properties. It also leads to a faster implementation which we exploit in numerical experiments validating the instability estimates and the stability of the constrained problem.   
\\

\\

\noindent%
\end{abstract}


\section{Introduction}\label{sec:introduction}
Determining the precise state of a quantum system is a difficult problem in experimental quantum physics. For light this was first established by Vogel and Risken \cite{Vogel:89} via homodyne detection and confirmed experimentally in \cite{Smithey:93}. Its mathematical relationship to the Radon transform lead to the name `quantum tomography'. By changing an experimental parameter equivalent to the angle in classical tomography, it is possible to reconstruct the full quantum state from multiple measurements on multiple equally prepared copies of the same quantum system. For electrons this technique has been established much more recently in \cite{Priebe:17} based on photon induced near field electron microscopy (PINEM). There, the interaction between an electron beam and two lasers with different frequencies and a precisely controllable phase shift mediated by a thin graphite film is used to achieve a sufficient number of different measurements to reconstruct the quantum state.  

Mathematically, a mixed quantum state can be described by a density operator \(\rho\), which is a positive semi-definite Hermitian operator with trace equal to \(1\) acting on a finite or infinite dimensional Hilbert space. The quantum tomography measurements are related to the density operator via a linear forward operator \(T\) which we will define precisely later. In \cite{Priebe:17} a new method called SQUIRRELS has been developed for the regularized reconstruction of the density operator from PINEM quantum tomography measurements. The mathematical analysis of the underlying inverse problem was done in \cite{Shi:20} which forms the basis for our work. We will fill the gaps by answering the open questions about stability and instability that were posed there.

We start by defining the relevant spaces and operators. We let \(\BopII\) be the space of Hilbert-Schmidt operators on \(\ell^2(\setZ)\) and let \((e_{l})_{l\in\setZ}\) be the canonical basis of \(\ell^2(\setZ)\). We frequently identify \(\BopII\) with \(\ell^{2}(\setZ\times\setZ)\) by considering for an operator \(\rho\) the scalar products \(\langle e_{l},\rho e_{k}\rangle\) as the corresponding sequence. We then let \(\BopIpsdtr\) be the subset of density operators, which are the positive semi-definite  operators with trace one. We furthermore define the unitary Fourier transform by \(\calF:L^{2}([-\pi,\pi])\rightarrow\ell^{2}(\setZ),f\mapsto (\frac{1}{\sqrt{2\pi}}\int_{-\pi}^{\pi}f(t)e^{-ikt}\dd t)_{k\in\setZ}\). 

As in \cite{Shi:20} we consider the linear forward operator \(T:\BopII\rightarrow L^{2}([-\pi,\pi]\times\setZ)\) given by
\begin{align}\label{al:op_stable1}
    (T\rho)(\theta,l)&=\langle e_{l},\Ucont\rho \Ucont^{*}e_{l}\rangle\notag\\
    \text{ with } \Ucont e_{l}&=\sum_{k\in\setZ}e^{i(k-l)\theta}J_{k-l}(2|g|)e_{k},
\end{align}
where \(g\) is a complex constant related related to the strength of the PINEM coupling. In this form, the special properties of the operator are not obvious and we therefore also introduce the following alternative formulation for \(\Ucont\)
\begin{align}\label{al:op_stable2}
    \Ucont&=\calF M_{a(\cdot+\theta)}\calF^{*}\\
    \text{ with }a&\in L^{2}([-\pi,\pi]),a(t):=\exp(2i|g|\sin(t)),\notag
\end{align}
which is based on modeling a phase modulation of the electron wave function \cite{Gaida:24}. The equivalence of both definitions is shown later in \Cref{lem:equiv_op}. We will see, that most elementary properties shown in \cite{Shi:20} are consequences of this decomposition of \(\Ucont\). In doing this we provide a more direct connection between the physical derivation of the operator and its mathematical properties.

After we defined the operator we formulate the constrained inverse problem
\begin{align}\label{al:problem_cont}
    \underset{\rho\in\BopIpsdtr}{\argmin}\|T\rho-\yobs\|_{L^{2}([-\pi,\pi)\times\setZ)}
\end{align}
with data \(\yobs\in L^{2}([-\pi,\pi]\times\setZ)\).

In \cite{Shi:20} it was shown, that \(T\) is injective and has an unbounded inverse. It was also shown, that if the operator is restricted to the set of density operators with suitable decay conditions on the off-diagonals, the problem becomes stable. However,  the setting described in  \eqref{al:problem_cont}, which is the natural formulation of the reconstruction problem, has not been analyzed yet and its stability was open until now.  However, it has been observed experimentally that solution algorithms which do not employ any explicit regularization lead to reasonable results as well \cite{Jeng:25}. The same practical observations also have been made for quantum tomography for photons, where in many popular methods only the restriction to valid density matrices is used without further regularization \cite{Fiuravsek:01,Bolduc:17,Strandberg:22}. 

We will show in \Cref{sec:stability} that for PINEM quantum tomography the constrained problem is already well-posed in the sense of Hadamard and thereby provide the mathematical underpinning for the observed results. As \(\BopIpsdtr\) is not compact this does not directly imply the existence of any global modulus of continuity, and we show in \Cref{sec:instability} that \(T^{-1}\) is indeed continuous but not uniformly continuous on \(T\BopIpsdtr\). This problem is therefore a rare example of a well-posed problem where still no global convergence rates exist for any possible regularization method. We then further investigate the instability of the problem and provide lower bounds for the cases with decay conditions on the off-diagonals complementing the upper bounds in \cite{Shi:20}. 
In \Cref{sec:discretization} we discuss how the choice of experimental parameters and matrix size influences the stability of the discrete problem and give a condition for its injectivity. There, we also derive a decomposition of the discrete operator which can be exploited for fast implementations using the Fast Fourier Transform. We conclude by conducting numerical experiments which underline the validity of our results in \Cref{sec:numerics}.
\section{Well-posedness}\label{sec:stability}
We first show the equivalence of definitions \eqref{al:op_stable1} and \eqref{al:op_stable2} of \(U_{\theta}\).
\begin{lemma}\label{lem:equiv_op}
For \(a\) as defined before
    \[\calF M_{a(\cdot+\theta)}\calF^{*}e_{l}=\sum_{k\in\setZ}e^{i(k-l)\theta}J_{k-l}(2|g|)e_{k}\]
    holds for all \(l\in\setZ\).
\end{lemma}
\begin{proof}
    We first use the Fourier transform and write \(a(t)=\sum_{k}(\calF a)_{k} \frac{1}{\sqrt{2\pi}}e^{ikt}\). From the properties of the Fourier transform and the Fourier convolution theorem it then follows that
    \begin{align*}
        \calF M_{a(\cdot+\theta)}\calF^{*}e_{l}&=\calF\left(\left(\sum_{k}(\calF a)_{k} \frac{1}{\sqrt{2\pi}}e^{ik(\cdot+\theta)}\right)\cdot(\calF^{*}e_{l})\right)\\
        &=\calF\left(\left(\sum_{k}e^{ik\theta}(\calF a)_{k} \calF^{*}e_{k}\right)\cdot(\calF^{*}e_{l})\right)\\
        &=\sum_{k}e^{ik\theta}(\calF a)_{k} \calF\left(\left(\calF^{*}e_{k}\right)\cdot(\calF^{*}e_{l})\right)\\
        &=\sum_{k}e^{ik\theta}(\calF a)_{k}\frac{1}{\sqrt{2\pi}}e_{k+l}=\sum_{k'}e^{i(k'-l)\theta}(\calF a)_{k'-l}\frac{1}{\sqrt{2\pi}}e_{k'}\\
    \end{align*}
    Now the identity \(e^{iz\sin(t)}=\sum_{k}J_{k}(z)e^{ikt}\) \cite[Eq.~9.1.41]{Abramowitz:65} gives us \((\calF a)_{k}=\sqrt{2\pi}J_{k}(2|g|)\), and inserting this into the previous equation results in the desired equality.
\end{proof}
Having shown that both representations are related in this way we can get the operator decomposition shown in \cite[Proposition 3.1]{Shi:20} in a slightly more general form. The indices indicate in which argument the operator is applied.
\begin{theorem}\label{the:decomp_cont}
For \(T:\BopII\rightarrow L^{2}([-\pi,\pi)\times\setZ)\) given by
\begin{align*}
    (T\rho)(\theta,l)&=\langle e_{l},\calF M_{c(\cdot+\theta)}\calF^{*} \rho\calF M_{\bar{c}(\cdot+\theta)}\calF^{*} e_{l}\rangle\\
\end{align*}
with \(c\in L^{2}([-\pi,\pi])\) the operator \(T\) can be decomposed as
\[T=\calF_{1}^{*}\calF_{2}M_{d_{c}}\calF^{*}_{2}\calG.\]
with
\begin{align*}
    d_{c}(k,t)&:=\frac{1}{\sqrt{2\pi}}\int_{-\pi}^{\pi}c(t-x)\overline{c(-x)}e^{ikx}\dd x 
\end{align*}
and \(\calG:\BopII\rightarrow\ell^{2}(\setZ\times\setZ),(\calG\rho)_{m,n}=\langle e_{n},\rho e_{n+m}\rangle\).
\end{theorem}
\begin{proof}
    We generally follow the proof in \cite{Shi:20} with slight modifications. First we get
    \begin{align*}
        (T\rho)(\theta,l)&=\frac{1}{2\pi}\sum_{m,n\in\setZ}e^{i(l-m)\theta}\jc{l-m}\rho_{m,n}e^{i(n-l)\theta}\overline{\jc{l-n}}\\
        &=\frac{1}{2\pi}\sum_{k\in\setZ}e^{ik\theta}\sum_{m\in\setZ}e^{i(l-m)}\jc{l-m}\rho_{m,k+m}\overline{\jc{l-k-m}}\\
        &=\frac{1}{2\pi}\sum_{k\in\setZ}e^{ik\theta}\left[(\rho_{\cdot,k+\cdot})*(\jc{\cdot}\overline{\jc{\cdot-k}})\right](l).
    \end{align*}
    We define \(\tilde{c}_{k}(x):=\overline{c(-x)}e^{ikx}\) such that \(\calF\tilde{c}_{k}=\overline{(\calF c)_{\cdot -k}}\).
    With the definition of \(\calG\) and the Fourier convolution theorem, we compute
    \begin{align*}
        (\calG\rho)(k,\cdot)*(\calF c \cdot\calF\tilde{c}_{k})=\calF\left[\calF^{*}((\calG\rho)(k,\cdot))\cdot(c *\tilde{c}_{k})\right].
    \end{align*}
    Inserting this into the previous formula gives us the desired equality.
\end{proof}
\begin{remark}
    In \cite{Shi:20} it was shown using further integral identities for Bessel functions that
    \[d_{a}(k,t)=\sqrt{2\pi}i^{k}e^{\frac{ikt}{2}}J_{k}\left(4|g|\sin\frac{t}{2}\right)\]
    for \(a\) as in definition \eqref{al:op_stable2} of the PINEM operator.
\end{remark}
In general, for operators of the given form it is sufficient to analyze the zeros of \(d_{c}\) to derive the relevant properties of the whole operator \(T\). For our further analysis in this section it is only required, that \(d_{a}\) has finitely many zeros for each \(k\) and that it is continuous.

To show stability we first show convergence of each off-diagonal separately. To simplify the argument we define for \(k\in\setZ\) the operators \(D_{k}:\ell^{2}(\setZ\times\setZ)\mapsto\ell^{2}(\setZ)\) by
\begin{align*}
    (D_{k}\rho)_{j}:=\rho_{j,j+k}.
\end{align*}
They are defined such that \(D_{k}\) maps to the \(k\)-th off-diagonal of \(\rho\), and enable us to formulate the following lemma.
\begin{lemma}\label{lem:off_diag_conv}
    For \((\rho^{(n)})_{n\in\setN}\subset\BopIpsdtr\) and \(\rho\in\BopIpsdtr\) such that \(T\rho^{(n)}\rightarrow T\rho\) in \(L^{2}([-\pi,\pi)\times\setZ)\),
    \begin{align*}
        \lim_{n\rightarrow\infty}\|D_{k}\rho^{(n)}-D_{k}\rho\|_{2}=0\quad\forall k\in\setZ.
    \end{align*}
\end{lemma}
\begin{proof}
We fix \(k\) and take \(\epsilon>0\). Then there exists a \(\delta>0\) such that the set \[A_{\delta,k}:=\{|d_{a}(k,\cdot)|<\delta\}\] has Lebesgue measure smaller than \(\epsilon\).
This follows immediately from the continuity of \(d_{a}(\cdot,k)\) combined with the fact that it has only finitely many discrete zeros for each fixed \(k\).

We then have
\begin{align*}
    \|D_{k}\rho^{(n)}-D_{k}\rho\|_{2}^{2}&=\int_{-\pi}^{\pi}|\calF^{*}(D_{k}\rho^{(n)})(s)-\calF^{*}(D_{k}\rho)(s)|^{2}\dd s\\
    &=\int_{A_{\delta,k}}|\calF^{*}(D_{k}\rho^{(n)})(s)-\calF^{*}(D_{k}\rho)(s)|^{2}\dd s\\
    &+\int_{[-\pi,\pi]\setminus A_{\delta,k}}|\calF^{*}(D_{k}\rho^{(n)})(s)-\calF^{*}(D_{k}\rho)(s)|^{2}\dd s
\end{align*}
From the positive semidefiniteness constraint and the AM-GM inequality, we get \[|\rho_{k,l}|\leq\sqrt{\rho_{k,k}\rho_{l,l}}\leq\frac{\rho_{k,k}+\rho_{l,l}}{2}.\]
We use this for the first summand together with the trace constraint and get
\begin{align*}
    \sum_{j}|(D_{k}\rho)|_{j}=\sum_{j}|\rho_{j,j+k}|\leq\sum_{j}\frac{\rho_{j,j}+\rho_{j+k,j+k}}{2}=1.
\end{align*}
This implies that \(D_{k}(\BopIpsdtr)\subset\ell^{1}(\setZ)\) and we can now use that the Fourier transform restricted to \(\ell^{1}(\setZ)\) is a bounded linear operator to \(C([-\pi,\pi))\) with norm \(1\). This leads to
\begin{align*}
    \int_{A_{\delta,k}}|\calF^{*}(D_{k}\rho^{(n)})(s)-\calF^{*}(D_{k}\rho)(s)|^{2}\dd s&\leq \epsilon \sup_{s\in A_{\delta,k}}|\calF^{*}(D_{k}\rho^{(n)})(s)-\calF^{*}(D_{k}\rho)(s)|^{2}\\
    &\leq \epsilon (\|\calF^{*}(D_{k}\rho^{(n)})\|_{\infty}+\|\calF^{*}(D_{k}\rho)\|_{\infty})^{2}\\
    &\leq \epsilon (\|D_{k}\rho^{(n)}\|_{1}+\|D_{k}\rho\|_{1})^2\leq 4\epsilon.
\end{align*}
For the second summand we take \(n\) large enough such that \(\frac{1}{\delta}\|T\rho^{(n)}-T\rho\|_{2}\leq\sqrt{\epsilon}\) and get 
\begin{align*}
    \epsilon&\geq\frac{1}{\delta^{2}}\|T \rho^{(n)}-T\rho\|_{2}^{2}\geq\frac{1}{\delta^{2}}\sum_{k'} \int_{-\pi}^{\pi}|d_{a}(s,k')|^{2}|\calF^{*}(D_{k'}\rho^{(n)})(s)-\calF^{*}(D_{k'}\rho)(s)|^{2}\dd s\\
    &\geq \frac{1}{\delta^{2}}\int_{[-\pi,\pi]\setminus A_{\delta,k}}|d_{a}(s,k)|^{2}|\calF^{*}(D_{k}\rho^{(n)})(s)-\calF^{*}(D_{k}\rho)(s)|^{2}\dd s\\
    &\geq \int_{[-\pi,\pi]\setminus A_{\delta,k}}|\calF^{*}(D_{k}\rho^{(n)})(s)-\calF^{*}(D_{k}\rho)(s)|^{2}\dd s
\end{align*}
Combined we get that for each \(k\) we can find \(N\) large enough, such that for all \(n\geq N\)
\[\|D_{k}\rho^{(n)}-D_{k}\rho\|_{2}^{2}\leq (1+4)\epsilon.\]
\end{proof}

It turns out that for \(\rho\in\BopIpsdtr\) off-diagonal-wise convergence is already enough to get convergence in \(\BopII\). For this we will use the following theorem of Riesz.
\begin{theorem}{\cite{Riesz:60},\cite[Corollary 5.5]{Elstrodt:02}}
    For \((f_{n})_{n\in\setN}\subset L^{p}\) and \(f_{n}\rightarrow f\in L^{p}\) point-wise almost everywhere the following two statements are equivalent.
    \begin{enumerate}
        \item \(||f_{n}||_{p}\rightarrow ||f||_{p}\)\\\item\(||f_{n}-f||_{p}\rightarrow 0.\)
    \end{enumerate}
\end{theorem}

This together with Pratt's lemma \cite{Pratt:60} enables us to prove an important lemma that underlines the regularizing effect of the positive semi-definiteness and trace constraint.
\begin{lemma}\label{lem:conv_density}
    For a sequence \((\rho^{(n)})_{n\in\setN}\subset\BopIpsdtr\) and \(\rho\in\BopIpsdtr\) such that 
    \[\lim_{n\rightarrow\infty}\|D_{k}\rho^{(n)}-D_{k}\rho\|_{2}=0\quad\forall k\in\setZ\]
    it holds that
    \[\rho^{(n)}\rightarrow\rho\]
    in \(\BopII\).
\end{lemma}
\begin{proof}
    We take a sequence \((\rho^{(n)})_{n\in\setN}\) and \(\rho\) with the required properties. Then by Riesz' theorem we get
    \[\lim_{n\rightarrow\infty}\|D_{k}\rho^{(n)}\|_{2}=\|D_{k}\rho\|_{2}\quad\forall k\in\setZ.\]

    From the positive semi-definiteness, we get the inequality \(|\rho_{j,k}|^{2}\leq \rho_{j,j}\rho_{k,k}\). Summing them up leads to
    \begin{align*}
        \|D_{k}\rho\|^{2}_{2}&=\sum_{j}|\rho_{j,j+k}|^{2}\leq\sum_{j}\rho_{j,j}\rho_{j+k,j+k}=:b_{k}\\
        \|D_{k}\rho^{(n)}\|^{2}_{2}&=\sum_{j}|\rho^{(n)}_{j,j+k}|^{2}\leq\sum_{j}\rho_{j,j}^{(n)}\rho_{j+k,j+k}^{(n)}=:b^{(n)}_{k}.
    \end{align*}
    We denote the shift operator that shifts forwards \(k\) times by \(S_{k}:\ell^{2}(\setZ)\rightarrow\ell^{2}(\setZ)\) and get
    \begin{align*}
        b^{(n)}_{k}&=\langle S_{k}D_{0}\rho^{(n)},D_{0}\rho^{(n)}\rangle_{\ell^{2}}\\
        b_{k}&=\langle S_{k}D_{0}\rho,D_{0}\rho\rangle_{\ell^{2}}.
    \end{align*}
    We then have for all \(k\in\setZ\)
    \begin{align*}
        |b^{(n)}_{k}-b_{k}|&=|\langle S_{k}D_{0}\rho^{(n)},D_{0}\rho^{(n)}\rangle_{\ell^{2}}-\langle S_{k}D_{0}\rho,D_{0}\rho\rangle_{\ell^{2}}|\\
        &\leq|\langle S_{k}(D_{0}\rho^{(n)}-D_{0}\rho),D_{0}\rho^{(n)}\rangle_{\ell^{2}}|+|\langle S_{k}D_{0}\rho,D_{0}\rho^{(n)}-D_{0}\rho\rangle_{\ell^{2}}|\\
        &\leq\|D_{0}\rho^{(n)}-D_{0}\rho\|_{\ell^{2}}\|D_{0}\rho^{(n)}\|_{\ell^{2}}+\|D_{0}\rho^{(n)}-D_{0}\rho\|_{\ell^{2}}\|D_{0}\rho\|_{\ell^{2}}\\
        &\leq\|D_{0}\rho^{(n)}-D_{0}\rho\|_{\ell^{2}}\|D_{0}\rho^{(n)}\|_{\ell^{1}}+\|D_{0}\rho^{(n)}-D_{0}\rho\|_{\ell^{2}}\|D_{0}\rho\|_{\ell^{1}}=2\|D_{0}\rho^{(n)}-D_{0}\rho\|_{\ell^{2}}.
    \end{align*}
    This means that for each \(k\) we have \(\lim_{n\rightarrow\infty}|b^{(n)}_{k}-b_{k}|=0\). Furthermore we get from the non-negativity of the diagonal entries of \(\rho\)
    \begin{align*}
        \sum_{k}|b_{k}|=\sum_{k}|\sum_{j}\rho_{j+k,j+k}\rho_{j,j}|=\sum_{j,k}\rho_{j+k,j+k}\rho_{j,j}=\left(\sum_{j}\rho_{j,j}\right)^{2}=1.
    \end{align*}
    In the same way we also get \(\sum_{k}|b^{(n)}_{k}|=1\) for all \(n\). Now we can view \(b^{(n)}\) as a sequence of sequences in \(\ell^{1}\) which converges point-wise to \(b\) which is again in \(\ell^{1}\). As the norms are always \(1\), we can apply Riesz' theorem and get that \(b^{(n)}\) converges to \(b\) in \(\ell^{1}\).

    Because \(b^{(n)}\) converges to \(b\) in \(\ell^{1}\) and \(\|D_{k}\rho^{(n)}\|^{2}_{2}\) converges to \(\|D_{k}\rho\|^{2}_{2}\) for all \(k\), by Pratt's lemma (see \cite{Pratt:60}) \(\|D_{k}\rho^{(n)}\|^{2}_{2}\) converges to \(\|D_{k}\rho\|^{2}_{2}\) in \(\ell^{1}\) meaning
    \begin{align*}
        \lim_{n\rightarrow\infty}\sum_{k}\left|\|D_{k}\rho^{(n)}\|^{2}_{2}-\|D_{k}\rho\|^{2}_{2}\right|=0.
    \end{align*}
    Now we can apply Riesz' theorem to get
    \begin{align*}
        &\lim_{n\rightarrow\infty}\sum_{k}\left|\|D_{k}\rho^{(n)}\|^{2}_{2}\right|=\sum_{k}\left|\|D_{k}\rho\|^{2}_{2}\right|,
    \end{align*}
    and we can furthermore calculate
    \begin{align*}
        \sum_{k}\left|\|D_{k}\rho\|^{2}_{2}\right|=\sum_{j,k}|\rho_{j,j+k}|^2=\sum_{j,k}|\rho_{j,k}|^2=\|\rho\|_{\BopII}^{2}.
    \end{align*}
    The same is again true for \(\rho^{(n)}\). Using the point-wise convergence of \(\rho^{(n)}\) to \(\rho\) and applying Riesz' theorem for a final time we then get that \(\rho^{(n)}\) converges to \(\rho\) in \(\BopII\).
\end{proof}
Combining lemmas \ref{lem:off_diag_conv} and \ref{lem:conv_density}  immediately gives us the remaining stability. 
\begin{theorem}\label{the:stability}
    For \((\rho^{(n)})_{n\in\setN}\subset\BopIpsdtr\) and \(\rho\in\BopIpsdtr\) such that \(T\rho^{(n)}\rightarrow T\rho\) in \(L^{2}([-\pi,\pi)\times\setZ)\),
    \begin{align*}
        \lim_{n\rightarrow\infty}\|\rho^{(n)}-\rho\|_{\BopII}=0.
    \end{align*}
\end{theorem}
Together with the injectivity shown in \cite{Shi:20} this implies that the problem \eqref{al:problem_cont} is well-posed for \(\yobs\in T\BopIpsdtr\).

From the structure of the proof we observe that we replaced the previously required decay conditions with a weaker decay condition which assumes that the norms of the off-diagonals form a sequence in \(\ell^{1}\). This weaker decay requirement then is directly fulfilled because of the constraints. As a trade-off this proof does not provide any rates and we investigate this in the next section.

\section{Instability}\label{sec:instability}
So far, we have established the stability or equivalently the continuity of the inverse when we restrict the operator to the set of density operators \(\BopIpsdtr\), but we have not shown any rates or determined the modulus of continuity. In this section, we will see that the inverse of the restricted operator is not uniformly continuous so no global rates or a global modulus of continuity exist. Additionally, we show instability estimates which complement the rates shown in \cite{Shi:20} with decay conditions.

First, we prove that the inverse is not uniformly continuous on \(T\BopIpsdtr\).
\begin{theorem}\label{the:not_uniform_cont}
   The operator \(T^{-1}:T\BopIpsdtr\rightarrow\BopIpsdtr\) is not uniformly continuous. This means there exists \(\epsilon>0\) such that for all \(\delta>0\) one can find \(\rho,\sigma\in \BopIpsdtr\) such that
   \[\|T\rho-T\sigma\|_{2}<\delta\text{ and } \|\rho-\sigma\|_{\calB_{2}}>\epsilon.\]
\end{theorem}
\begin{proof}
    Let \(0<\epsilon,\delta<\frac{1}{\sqrt{2}}\). Choose \(k\) such that \[\sqrt{\pi}\sup_{t\in[-\pi,\pi)}|d_{a}(k,t)|=\sqrt{2}\pi\sup_{s\in[0,4|g|]}|J_{k}(s)|<\delta\] and define \(\rho,\sigma\) by
    \begin{align*}
        \rho_{n,m}=\begin{cases}
            \frac{1}{2} \quad n=m=0\\
            \frac{1}{2} \quad n=m=k\\
            0 \quad\text{else}
        \end{cases}\quad\sigma_{n,m}=\begin{cases}
            \frac{1}{2} \quad n=0 \text{ and }(m=0 \text{ or }m=k)\\
            \frac{1}{2} \quad n=k \text{ and }(m=0 \text{ or }m=k)\\
            0 \quad\text{else.}
        \end{cases}\quad
    \end{align*}
    Then
    \[\|\rho-\sigma\|_{\calB_{2}}=\sqrt{2\frac{1}{2^2}}=\frac{1}{\sqrt{2}}> \epsilon\]
    and
    \begin{align*}
        \|T\rho-T\sigma\|_{2}^{2}&=\|T(\sigma-\rho)\|_{2}^{2}=\int_{-\pi}^{\pi}|d_{a}(-k,t)\frac{1}{2}e^{i\cdot0\cdot t}|^{2}\dd t+\int_{-\pi}^{\pi}|d_{a}(k,t)\frac{1}{2}e^{ikt}|^{2}\dd t\\
        &=\frac{1}{2}\int_{-\pi}^{\pi}|d_{a}(k,t)|^{2}\dd t\leq \pi\left(\sup_{t\in[-\pi,\pi)}|d_{a}(k,t)|\right)^{2}<\delta^{2}.
    \end{align*}
\end{proof}
This theorem shows that without decay conditions on the norms of the off-diagonals one cannot hope to get any stability estimates. This combination of well-posedness and missing uniform continuity of the inverse is highly unusual and provides a practical example for an inverse problem where well-posedness does not imply uniform convergence rates for any regularization method. This leads to the following corollary which is usually formulated for ill-posed inverse problems \cite[Proposition~3.11]{Engl:96} but still holds for our well-posed problem.
\begin{corollary}
    Let \(R_{\delta}:L^{2}([-\pi,\pi]\times\setZ)\rightarrow \BopII\) be any regularization method for problem \eqref{al:problem_cont}. Even though this problem is well-posed there can be no function \(\varphi:[0,\infty)\rightarrow[0,\infty)\) with \(\lim_{\delta\rightarrow 0}\varphi(\delta)=0\) such that
    \[E(\rho^{\dagger},\delta):=\sup\{\|R_{\delta}(y^{\delta})-\rho^{\dagger}\|_{2}:y^{\delta}\in T\BopIpsdtr,\,\|y^{\delta}-T\rho^{\dagger}\|_{2}\leq \delta\}\leq\varphi(\delta)\quad\forall \delta>0,\forall\rho^{\dagger}\in \BopIpsdtr.\]
\end{corollary}
\begin{proof}
Let \(\varphi:[0,\infty)\rightarrow[0,\infty)\) with \(\lim_{\delta\rightarrow 0}\varphi(\delta)=0\) then there exists \(\delta>0\) such that \(\varphi(\delta)<\frac{1}{2\sqrt{2}}\). By the proof of the previous theorem there exist \(\rho,\sigma\in\BopIpsdtr\) such that \(\|\rho-\sigma\|=\frac{1}{\sqrt{2}}\) and \(\|T\rho-T\sigma\|_{2}<\delta\). 
Now we derive lower bounds on \(E\) by considering the case \(y^{\delta}=T\sigma\). If \(\|R_{\delta}(T\sigma)-\sigma\|_{2}\geq \frac{1}{2\sqrt{2}}>\varphi(\delta)\) we immediately get that \(E(\sigma,\delta)\) cannot be bounded by \(\varphi(\delta)\). So we may assume that \(\|R_{\delta}(T\sigma)-\sigma\|_{2}\leq \frac{1}{2\sqrt{2}}\). Then we get by the triangle inequality
\begin{align*}
    E(\rho,\delta)\geq\|R_{\delta}(T\sigma)-\rho\|_{2}\geq-\|R_{\delta}(T\sigma)-\sigma\|_{2}+\|\sigma-\rho\|_{2}\geq \frac{1}{\sqrt{2}}-\frac{1}{2\sqrt{2}}=\frac{1}{2\sqrt{2}}>\varphi(\delta).
\end{align*}
\end{proof}

A remaining question is if the rates derived with decay conditions in \cite{Shi:20} are optimal. For this a more precise argument is necessary where we construct convergent sequences \((\rho^{(n)})_{n\in\setN}\subset \BopIpsdtr\) such that \(\rho^{(n)}\rightarrow\rho\in\BopIpsdtr\) much slower than \(T\rho^{(n)}\rightarrow T\rho\). One of the challenges in these constructions lies in ensuring that all sequence elements are valid density operators. To ensure positive semi-definiteness we use the following lemma which is a straightforward extension of the Gershgorin circle theorem \cite[Theorem 1.1]{Varga:11}.
\begin{lemma}\label{lem:gershgorin}
    Let \(\rho\in\calB_{2}\) be hermitian. If for all \(j\in\setZ\)
    \[\rho_{jj}\geq \sum_{k\in \setZ,k\neq j}|\rho_{jk}|,\]
    then \(\rho\) is positive semidefinite.
\end{lemma}
Now we construct sequences of operators by keeping their diagonal fixed and choosing the off-diagonal entries as Fourier coefficients of hat functions.

For \(n>0\) we define the hat function \(h^{(n)}\in L^{2}((-\pi,\pi])\) by
\begin{align*}
    h^{(n)}(x)=\begin{cases}1-n|x|\quad \text{ for } |x|\leq\frac{1}{n}\\0\quad\text{else.}\end{cases}
\end{align*}
Its Fourier series is given by
\begin{align*}
    \hat{h}^{(n)}(l)&=\frac{1}{\sqrt{2\pi}}\int_{-\frac{1}{n}}^{\frac{1}{n}}(1-n|x|)e^{-ilx}\dd x=\frac{1}{\sqrt{2\pi}}\int_{-\frac{1}{n}}^{\frac{1}{n}}(1-n|x|)\cos(lx)=\frac{2}{\sqrt{2\pi}}\int_{0}^{\frac{1}{n}}(1-nx)\cos(lx)\\
    &=\begin{cases}\frac{2n}{\sqrt{2\pi}}\frac{1-\cos(\frac{l}{n})}{l^{2}}
    \quad l\in\setZ\setminus\{0\}\\
    \frac{1}{\sqrt{2\pi}}\frac{1}{n}\quad l=0.\end{cases}
\end{align*}
Now we define for \(k\in\setN,k>0\) and \(0<\epsilon<1\) the operator \(\eta(n,k,\epsilon)\) by
\begin{align*}
    \eta(n,k,\epsilon)_{l+k,l}=\eta(n,k,\epsilon)_{l,l+k}=\frac{1}{2}(2|k|)^{-1-\epsilon}n^{-\epsilon}\hat{h}^{(n)}(l).
\end{align*}
This operator is symmetric and only contains values on two off-diagonals which are given by scaled versions of \(\hat{h}^{(n)}\). 

We now construct a sequence for the diagonal to get a positive semidefinite operator. For this we define the diagonal operator
\begin{align*}
    \tau(\epsilon)_{l,l}:=\begin{cases}\frac{4}{\sqrt{2\pi}}|l|^{-(1+\epsilon)}
    \quad l\in\setZ\setminus\{0\}\\
    \frac{4}{\sqrt{2\pi}}\quad l=0.\end{cases}
\end{align*}
Now we combine them and define
\begin{align*}
    \sigma(n,k,\epsilon)&:=\frac{1}{\tr\tau(\epsilon)}(\tau(\epsilon)+\eta(n,k,\epsilon))\\
    \rho(\epsilon)&:=\frac{1}{\tr\tau(\epsilon)}\tau(\epsilon).
\end{align*}
Their properties are summarized in the following lemma.
\begin{lemma}\label{lem:seq_prop}
    For \(k\in\setN\) and \(n,\epsilon>0\) the operators \(\sigma(n,k,\epsilon),\rho(\epsilon)\) are density operators. Furthermore, the following bounds hold 
    \begin{align*}
        \|\sigma(n,k,\epsilon)-\rho(\epsilon)\|_{\calB_{2}}&=\frac{1}{\tr\tau(\epsilon)}\frac{1}{\sqrt{3}}(2|k|)^{-1-\epsilon}n^{-\frac{1}{2}-\epsilon},\\
        \|T\sigma(n,k,\epsilon)-T\rho(\epsilon)\|_{2}&\leq\frac{1}{\tr\tau(\epsilon)}2\sqrt{\pi}\frac{(2k)^{-1-\epsilon}|g|^{k}}{k!\sqrt{(2k+1)(2k+2)(2k+3)}}n^{-k-\frac{1}{2}-\epsilon},\\
        \sum_{l\in\setZ}|\sigma(n,k,\epsilon)_{l+k,l}|&=\frac{1}{\tr\tau(\epsilon)}\frac{1}{2}(2|k|)^{-1-\epsilon}n^{-\epsilon}.
    \end{align*}
\end{lemma}
The proof which mainly consists of simple calculations can be found in \Cref{sec:proof_seq}.

With these bounds we can now show instability estimates. To simplify the arguments we make use of Landau notation and write \(f\in\Omega(h)\) for \(h\in\calO(f)\).

We start with the band-limited case where we show that the Hölder exponent from \cite{Shi:20} cannot be improved.
\begin{theorem}\label{the:bandlimited}
    For \(k\in\setN\), \(\mu>\frac{1}{1+2k}\) and every \(C>0\) there exist \(\rho,\sigma\in\calT\) with \(\rho,\sigma\) band-limited by \(k\) such that
    \[\|\sigma-\rho\|_{\calB_2}\geq C\|T\sigma-T\rho\|_{2}^{\mu}\]
\end{theorem}
\begin{proof}
    For \(k\) and \(\epsilon>0\) fixed we have
\begin{align*}
    \|\sigma(n,k,\epsilon)-\rho(\epsilon)\|_{\calB_{2}}&\in \Omega(n^{-\frac{1}{2}-\epsilon}),\\
    \|T\sigma(n,k,\epsilon)-T\rho(\epsilon)\|_{2}^{\mu}&\in \calO(n^{\mu(-k-\frac{1}{2}-\epsilon)}).
\end{align*}
Now we choose \(\epsilon\) such that \(\frac{1+2\epsilon}{1+2k+2\epsilon}<\mu\) and get
\begin{align*}
    \mu\left(k+\frac{1}{2}+\epsilon\right)=\frac{1}{2}\mu(1+2k+2\epsilon)>\frac{1}{2}(1+2\epsilon).
\end{align*}
This implies \(\|T\sigma(n,k,\epsilon)-T\rho(\epsilon)\|_{2}^{\mu}\in o(n^{-\frac{1}{2}-\epsilon})\) and the claimed inequality  holds for all \(n\) large enough.
\end{proof}

For exponential decay we get the following lower bound.

\begin{theorem}\label{the:exponential}
    For \(0<b<1\) and every \(C>0\) there exist \(\rho,\sigma\in\calT\) with \(\rho,\sigma\) having exponentially decaying off-diagonals, meaning
    \begin{align*}
        \sum_{l\in\setZ}|\rho_{l,l+k}|\leq \tilde{C}b^{|k|},\quad\sum_{l\in\setZ}|\sigma_{l,l+k}|\leq \tilde{C}b^{|k|}
    \end{align*}
    such that
    \[\|\sigma-\rho\|_{\calB_{2}}\geq C\Phi(\|T\sigma-T\rho\|_{2})^{-\alpha}\text{ with }\quad \Phi(t):=e^{\sqrt{(-\ln t)(-\ln b)}}\]
    for \(\alpha>\sqrt{2}\).
\end{theorem}
\begin{proof}
    For \(k>0\) we set \(n(k):=b^{-\frac{1}{\epsilon}k}\) to get
    \begin{align*}
        \sum_{l\in\setZ}|\sigma(n(k),k,\epsilon)_{l+k,l}|\in \calO(b^{k})
    \end{align*}
    by \Cref{lem:seq_prop}. Therefore, the off-diagonals decay with the desired rate. For the norms we get by the same lemma
    \begin{align*}
        \|\sigma(n(k),k,\epsilon)-\rho(\epsilon)\|_{\calB_{2}}&\in \Omega(b^{\frac{1}{\epsilon}k(\frac{1}{2}+\epsilon)\nu})\quad \forall \nu>1\\
    \|T\sigma(n(k),k,\epsilon)-T\rho(\epsilon)\|_{2}&\in o(b^{\frac{1}{\epsilon}k^{2}}).
    \end{align*}
    We now compute
    \begin{align*}
        \Phi(b^{\frac{1}{\epsilon}k^{2}})^{-\alpha}=e^{-\alpha\sqrt{(-\ln b)(-\ln b^{\frac{1}{\epsilon}k^{2}})}}=e^{-|\ln b|\alpha\frac{1}{\sqrt{\epsilon}}k}=b^{\alpha\frac{1}{\sqrt{\epsilon}}k}.
    \end{align*}
    Then for \(\epsilon=\frac{1}{2}\) and \(1<\nu<\frac{\alpha}{\sqrt{2}}\) we have
    \begin{align*}
        \frac{1}{\epsilon}(\frac{1}{2}+\epsilon)\nu=\sqrt{2}\sqrt{2}\nu<\sqrt{2}\alpha=\frac{1}{\sqrt{\epsilon}}\alpha.
    \end{align*}
    This implies
    \begin{align*}
        \Phi(\|T\sigma(n(k),k,\epsilon)-T\rho(\epsilon)\|_{2})^{-\alpha}\in o(b^{\frac{1}{\epsilon}k(\frac{1}{2}+\epsilon)\nu}),
    \end{align*}
    and for \(k\) large enough the desired inequality holds.
\end{proof}
This bound has the same form as the one in \cite{Shi:20} but there the constant \(\alpha\) is equal to \(1\).

For the final case of polynomial decay, the previous construction is not necessary as the factorial term in the bound of \(\|T \sigma(n, k, \epsilon) - T \rho(\epsilon)\|_{2}\) in \Cref{lem:seq_prop} dominates if \(n\) is chosen appropriately to get polynomial decay. We instead use the following simpler construction similar to the one in \Cref{the:not_uniform_cont}.  For \(\mu>0\) we define
\begin{align*}
    \tilde{\rho}(\mu)_{j,j}&=\begin{cases}(1-2^{-\mu}) \quad j=0\\
    (1-2^{-\mu})2^{-\mu l} \quad j=2^{l},l\in\setZ,l> 0\\
    0\quad \text{else}\end{cases}\\
    \tilde{\eta}(k,\mu)_{l,m}&=\begin{cases}(1-2^{-\mu})2^{-\mu k} \quad m=0,l=2^{k} \text{ or }m=2^{k},l=0\\
    0\quad \text{else}\end{cases}\\
    \tilde{\sigma}(k,\mu)&=\tilde{\eta}(k,\mu)+\tilde{\rho}(\mu).
\end{align*}
By \Cref{lem:gershgorin} we get \(\tilde{\sigma}(k,\mu)\in\BopIpsdtr\) and we can now show the following bound for polynomial decay.
\begin{theorem}\label{the:polynomial}
    For \(\alpha>\mu>0\) and every \(C>0\) there exist \(\rho,\sigma\in\BopIpsdtr\) with \(\rho,\sigma\) having polynomially decaying off-diagonals, meaning
    \begin{align*}
        \sum_{l\in\setZ}|\rho_{l,l+k}|\leq \tilde{C}|k|^{-\mu},\quad\sum_{l\in\setZ}|\sigma_{l,l+k}|\leq \tilde{C}|k|^{-\mu}
    \end{align*}
    such that
    \[\|\sigma-\rho\|_{\calB_{2}}\geq C\tilde{\Phi}(\|T\sigma-T\rho\|_{2})^{-\alpha}\text{ with }\quad\tilde{\Phi}(t):=\frac{-\ln t}{\ln(-\ln t)}.\]
\end{theorem}
\begin{proof}
    By construction, \(\tilde{\sigma}(k,\mu)\) fulfills the polynomial decay condition. The norm differences are bounded by
    \begin{align*}
        \|\tilde{\sigma}(k,\mu)-\tilde{\rho}(\mu)\|_{\calB_{2}}&=\sqrt{2}(1-2^{-\mu})2^{-\mu k}\in \Omega(2^{-\mu k}),\\
    \|T\tilde{\sigma}(k,\mu)-T\tilde{\rho}(\mu)\|_{2}&\leq2\pi\sqrt{2}(1-2^{-\mu})\frac{g^{2^{k}}}{(2^{k})!}2^{-\mu k}\in o((2^{k})^{-\nu 2^{k}}) \quad \forall\nu\in(0,1).
    \end{align*}
    We then compute
    \begin{align*}
        \tilde{\Phi}((2^{k})^{-\nu 2^{k}})^{-\alpha}=\left(\frac{\ln(-\ln((2^{k})^{-\nu 2^{k}}))}{-\ln((2^{k})^{-\nu 2^{k}})}\right)^{-\alpha}=\left(\frac{\ln\nu+k\ln2+\ln k+\ln\ln 2}{\nu 2^{k}k\ln2}\right)^{-\alpha}\in\calO(2^{-\alpha k}).
    \end{align*}
    Finally, because \(\alpha>\mu\) we have \(\tilde{\Phi}( \|T\tilde{\sigma}(k,\mu)-T\tilde{\rho}(\mu)\|_{2})^{-\alpha}\in o(2^{-\mu k})\) and we can choose \(k\) large enough such that the claimed inequality holds.
\end{proof}
These bounds again have the same form as the ones in \cite{Shi:20}, where a rate with \(\alpha=\mu-\frac{1}{2}\) was derived for \(\mu>\frac{1}{2}\). For the case \(0<\mu\leq\frac{1}{2}\) there is no upper bound so far.

An overview of the instability results compared to their counterparts from \cite{Shi:20} is given in \Cref{tab:stability}.
\begin{table}[h]
    \centering
    \begin{tabular}{lccc}
    \toprule
        decay type&\(\Phi(t)\)&\(\alpha\) stability&\(\alpha\) instability\\
\midrule
         band limited&\(t^{-1}\)&\(\frac{1}{1+2k}\)&\(>\frac{1}{1+2k}\)  \\

         exponential&\(e^{\sqrt{(-\ln t)(-\ln b)}}\)&\(1\)&\(>\sqrt{2}\) \\

         polynomial&\(\frac{-\ln t}{\ln(-\ln t)}\)&\(\mu-\frac{1}{2}>0\)&\(>\mu\) \\
         \bottomrule
    \end{tabular}
    \caption{Summary of stability results from \cite{Shi:20} and new instability results. The rate in each case is given by \(\Phi(t)^{-\alpha}\).}
    \label{tab:stability}
\end{table}

\section{Discretization}\label{sec:discretization}
There are different ways to derive discretizations of the operator \(T\) depending on how the unitary operators \(\Ucont\) are discretized. We consider the simplest one where \(\Ucont\) is restricted to a finite dimensional subspace by truncation and the set of measurement angles is given by \(\Theta\). 
\begin{definition}\label{def:disc_op_trunc}
    For \(\theta\in \Theta\) and \(N\in\setN\) define \(\trunc{U}_{N,\theta} :\setC^{N}\rightarrow\setC^{N}\) by
    \begin{align*}
        (\trunc{U}_{N,\theta} )_{k,l}&:=e^{i(k-l)\theta}J_{k-l}(2|g|)
    \end{align*}
    and
    \(\trunc{T}_{N,\Theta}:\Herm(N)\rightarrow \setR^{|\Theta|}\times\setR^{N}\)
    by
    \begin{align*}
        (\trunc{T}_{N,\Theta}\rho)(\theta,l)&=\langle e_{l},\trunc{U}_{N,\theta} \rho \trunc{U}_{N,\theta} ^{*}e_{l}\rangle\quad 0\leq l<N,\,\theta\in\Theta.
    \end{align*}
\end{definition}
Alternatively, this can be described as extending \(\rho\) by zero to get an infinite dimensional density operator, then applying \(T\) and finally restricting to a finite dimensional subspace and a finite set of angles.

This discretization has the disadvantage, that the discrete operators \(\trunc{U}_{N,\theta}\)  are no longer unitary, but it allows us to carry over many of the properties of the continuous operator to the discrete setting. For this we define the unitary discrete Fourier transform \(F_{N}:\setC^{N}\rightarrow\setC^{N}\) by
\begin{align*}
    (F_{N}v)_{k}:=\frac{1}{\sqrt{N}}\sum_{j=0}^{N-1}v_{j}e^{-i\frac{2\pi}{N}jk}.
\end{align*}
To interface between the discrete Fourier transform and Fourier series we define for a finite set \(\calI\subset\setZ\) the embedding operator \(Q_{\calI}:\setC^{|\calI|}\rightarrow\ell^{2}(\setZ)\) by
\begin{align*}
    (Q_{\calI}v)_m:=\begin{cases}
        v_{m}\quad\text{for}\quad m \in \calI\\
        0\quad\text{else}.
    \end{cases}
\end{align*}
Here we assume that \(v\) is indexed by \(\calI\). To simplify the notation when applying operators to two dimensional objects, we indicate along which argument the operator is applied by numbers in brackets e.g. \(F_{N}^{(2)}\) is the Fourier transform in the second argument. To further simplify the indexing we introduce the following index sets and a bijection between them: 
\begin{align*}
    \calI_{N}&:=\{0,\dots,2N-2\}\\
    \calJ_{N}&:=\{-N+1,\dots,N-1\}\\
    \varphi_{N}:\calI_{N}\rightarrow\calJ_{N},\,
    \varphi_{N}(n)&:=\begin{cases}
        n\quad0\leq n< N\\
        n-(2N-1)\quad N\leq n< 2N-1.
        \end{cases}
\end{align*}
Now we are ready to derive the decomposition of the discrete operator.
\begin{theorem}\label{the:decomp_trunc}
The discrete operator can be decomposed as
    \[(\trunc{T}_{N,\Theta}\rho)(\theta,l)=\calF^{*(1)}(Q_{\calJ_{N}}^{(1)}F_{2N-1}^{(2)}M_{\trunc{d}_{N}}F_{2N-1}^{*(2)}\trunc{G}_{N}(\rho))(\theta,l) \quad0\leq l<N,\,\theta\in[-\pi,\pi].\]
with
\begin{align*}
    \trunc{d}_{N}(k,l)&:=\sqrt{2N-1}F_{2N-1}^{*(2)}(K(2|g|)_{k,\cdot})(l)\quad (k,l)\in\calJ_{N}\times\calI_{N}\\
    K(2|g|)_{k,l}&:=\sqrt{2\pi}J_{\varphi_{N}(l)-k}(2|g|)J_{\varphi_{N}(l)}(2|g|)\quad (k,l)\in\calJ_{N}\times\calI_{N}
\end{align*}
and \(\trunc{G}_{N}:\Herm(N)\rightarrow (\calJ_{N}\times\calI_{N}\rightarrow \setC),\)
\begin{align*}
    (\trunc{G}_{N}\rho)_{m,n}:=\begin{cases}
        \rho_{n,m+n}\quad\text{for}\quad 0\leq m+n,n <N\\
        0\quad\text{else.}
    \end{cases}
\end{align*}
\end{theorem}
\begin{proof}
    We follow the proof of \Cref{the:decomp_cont} for the continuous case and calculate
    \begin{align*}
        (\trunc{T}_{N,\Theta}\rho)(\theta,l)&=\sum_{j,k=0}^{N-1}e^{i(l-j)\theta}\jv{l-j}\rho_{j,k}e^{i(k-l)\theta}\jv{l-k}\\
        &=\sum_{n=-N+1}^{N-1}e^{in\theta}\sum_{j=\max(0,-n)}^{N-1-\max(0,n)}\jv{l-j}\jv{l-(n+j)}\rho_{j,n+j}\\
        &=\sum_{n=-N+1}^{N-1}e^{in\theta}\sum_{j=0}^{N-1}\jv{l-j}\jv{l-(n+j)}(\trunc{G}_{N}\rho)_{n,j}.\\
    \end{align*}
    The second sum can again be seen as a convolution. We extend it to get a cyclic convolution as follows and use the discrete Fourier convolution theorem
    \begin{align*}
        \sum_{j=0}^{N-1}\jv{l-j}&\jv{l-(n+j)}(\trunc{G}_{N}\rho)_{n,j}=\left[(\jv{\varphi_{N}(\cdot)}\jv{\varphi_{N}(\cdot)-n})\cycconv (\trunc{G}_{N}\rho)_{n,\cdot}\right](l)\\
        &=\sqrt{2N-1}F_{2N-1}\left[F_{2N-1}^{*}(\jv{\varphi_{N}(\cdot)}\jv{\varphi_{N}(\cdot)-n})\cdot F_{2N-1}^{*}(\trunc{G}_{N}\rho)_{n,\cdot}\right](l)\\
        &=\frac{1}{\sqrt{2\pi}}F_{2N-1}\left[\trunc{d}_{N}(n,\cdot)\cdot F_{2N-1}^{*}(\trunc{G}_{N}\rho)_{n,\cdot}\right](l).
    \end{align*}
    Inserting this into the previous equation results in the desired formula.
\end{proof}
We can use the Shannon sampling theorem \cite{Frazier:94} for the last Fourier transform as the Fourier series only has finitely many non-zero coefficients (one for each off-diagonal) and is therefore uniquely determined by \(M\geq 2N-1\) equidistant angles. As a corollary, we get a condition which guarantees injectivity.
\begin{corollary}
    For \(M\geq 2N-1\) and \(g\in \setR\), such that \(\trunc{d}_{N}\) is nonzero everywhere the operator \(\trunc{T}_{N,\Theta}\) is injective.
\end{corollary}
\begin{proof}
    Because of \(M\geq 2N-1\) the operator \(E_{\Theta}\calF^{*,(1)}Q_{\calJ_{N}}^{(1)}\), where \(E_{\Theta}\) is the operator that evaluates on all points in \(\Theta\), is injective by the sampling theorem. All other operators that comprise \(\trunc{T}_{N,\Theta}\) apart from \(M_{\trunc{d}_{N}}\) are injective. As a multiplication operator \(M_{\trunc{d}_{N}}\) is injective if and only if \(\trunc{d}_{N}\) is non-zero everywhere.
\end{proof}
Using this corollary the injectivity for a given \(g\) and \(N\) can then be checked easily by computing \(\trunc{d}_{N}\). In general, because \(\trunc{d}_{N}\) is holomorphic in \(|g|\) it has only isolated zeros and one can expect it to be non-zero in practice. In a similar way to our previous instability estimates, we can derive a lower bound on the inverse of the discretized operator. 

\begin{theorem}
    If \(g\), \(N\) and the number of angles \(M\) are chosen such that \(\trunc{T}_{N,\Theta}\) is invertible, then
    \begin{align*}
        \|\trunc{T}_{N,\Theta}^{-1}\|_{2}\geq\frac{((N-1)!)^{2}}{|g|^{N-1}\sqrt{(2N-2)!\cdot M}}
    \end{align*}
\end{theorem}
\begin{proof}
Similar to the construction in \Cref{the:not_uniform_cont} we consider the hermitian matrix 
\begin{align*}
   \eta(N)_{m,n}:=\begin{cases}
    1\quad (m=0\text{ and } n=N-1)\text{ or }(n=N-1\text{ and } n=0)\\
    0\quad \text{ else.}
\end{cases} 
\end{align*}
If the inverse of \(T_{N,\Theta}\) exists, its norm is bounded from below as follows
\begin{align*}
    \|T_{N,\Theta}^{-1}\|=\sup_{y\in T_{N,\Theta}(\Herm(N))\setminus\{0\}}\frac{\|T_{N,\Theta}^{-1}y\|_{2}}{\|y\|_{2}}\geq \frac{\|\eta(N)\|_{2}}{\|T_{N,\Theta}\eta(N)\|_{2}}=\frac{\sqrt{2}}{\|T_{N,\Theta}\eta(N)\|_{2}}.
\end{align*}
We now compute \(T_{N,\Theta}\eta(N)\) by directly using \Cref{def:disc_op_trunc}
\begin{align*}
    (T_{N,\Theta}\eta(N))(l,\theta)&=\sum_{m,n=0}^{N-1}e^{i(l-m)\theta}\jv{l-m}\eta(N)_{m,n}e^{i(n-l)\theta}\jv{l-n}\\&=e^{i(N-1)\theta}\jv{l}\jv{l-(N-1)}+e^{-i(N-1)\theta}\jv{l-(N-1)}\jv{l}\\
    &=2\cos((N-1)\theta)\jv{l-(N-1)}\jv{l}.
\end{align*}
The squared norm \(\|T_{N,\Theta}\eta(N)\|_{2}^{2}\) can be split into two factors
\begin{align*}
    \|T_{N,\Theta}\eta(N)\|_{2}^{2}=\left(\sum_{\theta\in\Theta}4\cos((N-1)\theta)^{2}\right)\left(\sum_{l=0}^{N-1}(\jv{l-(N-1)}\jv{l})^{2}\right)
\end{align*}
By inserting the assumption that \(M\) equidistant angles \(\theta\) are used and that \(M\geq 2N-1\) we can compute the first factor
\begin{align*}
    \sum_{m=0}^{M-1}4\cos\left((N-1)\frac{2\pi m}{M}\right)^{2}&=2\sum_{m=0}^{M-1}\left(1+\cos\left((N-1)\frac{4\pi m}{M}\right)\right)=2M
\end{align*}
as the sum over the cosines vanishes. We bound the second factor by using the bound for the Bessel functions \cite[Eq.~10.14.4]{Lozier:03} \[J_{n}(2x)\leq\frac{x^{|n|}}{|n|!} \quad n\in\setN.\]
This leads to 
\begin{align*}
    \sum_{l=0}^{N-1}(\jv{l-(N-1)}\jv{l})^{2}&\leq\sum_{l=0}^{N-1}\left(\frac{|g|^{|l-(N-1)|}}{|l-(N-1)|!}\frac{|g|^{|l|}}{|l|!}\right)^{2}=\sum_{l=0}^{N-1}\left(\frac{|g|^{N-1}}{((N-1)-l)!l!}\right)^{2}\\
    &=\frac{|g|^{2(N-1)}}{((N-1)!)^{2}}\sum_{l=0}^{N-1}\binom{N-1}{l}^{2}=\frac{|g|^{2(N-1)}}{((N-1)!)^{2}}\binom{2(N-1)}{N-1}\\
    &=\frac{|g|^{2(N-1)}(2(N-1))!}{((N-1)!)^{4}}.
\end{align*}
Combining the bounds for both factors and taking the square root results in the claimed lower bound of the inverse.
\end{proof}
Due to the squared factorial this bound increases super exponentially showing the numerical instability of the unconstrained discrete problem. However, this increase can be delayed by the choice of a large enough \(|g|\) and a rough numerical approximation suggests, that the value \(|g|\) has to be at least of size \(\frac{N}{5}\) to prevent a blowup of the lower bound in the previous theorem. This is only a heuristic, as we just have a lower bound and the norm of \(\trunc{T}_{N,\Theta}^{-1}\) could still be much larger.

\section{Numerical examples}\label{sec:numerics}
We confirm our theoretical results by testing them in computational examples. The code for their reproduction is available at \cite{Oberender:26stabdata}. Throughout this section, we will use the discretization discussed in \Cref{sec:discretization} and try to avoid discretization artifacts by choosing sufficiently large matrix sizes. For the instability results, we use the constructed example sequences from \Cref{sec:instability} and compute the difference of the norms of the matrices and the corresponding data. To underline the stability we consider the inverse problem and we investigate the reconstruction qualities without regularization. To be able to perform our calculations efficiently, we use the decomposition from \Cref{the:decomp_trunc} together with the fast Fourier transform for the evaluation of the operator \(T\) and its adjoint. This reduces the asymptotic complexity from \(\calO(MN^{2})\) to \(\calO(MN(\log M+\log N))\) for an \(N\times N\) matrix and \(M\) different measurement angles.

\subsection{Numerical validation of instability estimates}
 To further reduce the computational effort we consider the operator \(\tilde{T}_{N}=M_{\trunc{d}_{N}}F_{2N-1}^{*(2)}\trunc{G}_{N}(\rho)\) analogously to the continuous case as the final two Fourier transforms do not impact the stability. For all examples we use matrices of size \(N=3000\) which is much larger than any physically attainable examples, but it is necessary to lower the impact of discretization errors.
We start with the simple example
\begin{align*}
        \rho_{n,m}(k)=\begin{cases}
            \frac{1}{2} \quad n=m=0\\
            \frac{1}{2} \quad n=m=k\\
            0 \quad\text{else}
        \end{cases}\quad\sigma_{n,m}(k)=\begin{cases}
            \frac{1}{2} \quad n=0 \text{ and }(m=0 \text{ or }m=k)\\
            \frac{1}{2} \quad n=k \text{ and }(m=0 \text{ or }m=k)\\
            0 \quad\text{else}
        \end{cases}\quad
\end{align*}
from the proof of \Cref{the:not_uniform_cont} to show that general rates without additional constraints on the domain are not possible. In \Cref{fig:gen_band}(a) we plot the errors \(\|\tilde{T}_{N}\rho(k)-\tilde{T}_{N}\sigma(k)\|_{2}\) for different values of \(|g|\) and observe that after an initial plateau which is dependent on \(|g|\) the norm of the results decrease rapidly while the norm of \(\|\rho(k)-\sigma(k)\|_{F}\) stays constant by definition. 
\begin{figure}
\begin{subfigure}[b]{0.45\textwidth}
    \centering
\begin{tikzpicture}[scale=0.75]

\definecolor{crimson2143940}{RGB}{214,39,40}
\definecolor{darkgray176}{RGB}{176,176,176}
\definecolor{darkorange25512714}{RGB}{255,127,14}
\definecolor{forestgreen4416044}{RGB}{44,160,44}
\definecolor{lightgray204}{RGB}{204,204,204}
\definecolor{mediumpurple148103189}{RGB}{148,103,189}
\definecolor{steelblue31119180}{RGB}{31,119,180}

\begin{axis}[
legend cell align={left},
legend style={
  fill opacity=0.8,
  draw opacity=1,
  text opacity=1,
  at={(0.03,0.03)},
  anchor=south west,
  draw=lightgray204
},
log basis y={10},
tick align=outside,
tick pos=left,
x grid style={darkgray176},
xlabel={\(\displaystyle k\)},
xmin=-7.75, xmax=206.75,
xtick style={color=black},
y grid style={darkgray176},
ylabel={\(\displaystyle\|\tilde{T}\rho(k)-\tilde{T}\sigma(k)\|_{2}\)},
ymin=1e-100, ymax=261238.172498916,
ymode=log,
ytick style={color=black},
ytick={1e-80,1e-60,1e-40,1e-20},
yticklabels={
  \(\displaystyle {10^{-80}}\),
  \(\displaystyle {10^{-60}}\),
  \(\displaystyle {10^{-40}}\),
  \(\displaystyle {10^{-20}}\)
}
]
\addplot [thick, steelblue31119180]
table {%
7 0.00423036106871031
12 0.00113240045961861
17 1.2939959040385e-05
22 2.5553848694568e-08
27 1.45191767944581e-11
32 3.10278617787256e-15
37 2.95277681859362e-19
42 1.4071919714563e-23
47 3.66274334610317e-28
52 5.56769512280118e-33
57 5.21353674744032e-38
62 3.1412140151151e-43
67 1.26277061918992e-48
72 3.49254866369707e-54
77 6.82308958707175e-60
82 9.63280230245641e-66
87 1.00262350994557e-71
92 7.83075241045414e-78
97 4.66198783584407e-84
102 2.1456091814668e-90
107 7.73119013510208e-97
112 2.20619303557056e-103
117 5.03823573578186e-110
122 9.29602184856778e-117
127 1.39798758891418e-123
132 1.72742681080359e-130
137 1.76694120049286e-137
142 1.50650065053188e-144
147 1.07752660158482e-151
152 6.50419318794909e-159
157 0
162 0
167 0
172 0
177 0
182 0
187 0
192 0
197 0
};
\addlegendentry{g=2.7}
\addplot [thick, darkorange25512714]
table {%
2 0.0033686511883776
7 0.00295044121039859
12 0.00270810315549545
17 0.00255702343673278
22 0.00229243608530633
27 0.00233474554415871
32 0.0020941482646802
37 0.00224261560620726
42 0.00055952645533452
47 3.84166040425074e-05
52 1.07953948546133e-06
57 1.50173538544149e-08
62 1.1595112032842e-10
67 5.37242063358438e-13
72 1.58201724812588e-15
77 3.09522601381882e-18
82 4.16954372038433e-21
87 3.98243042588344e-24
92 2.76435220151848e-27
97 1.42426838136195e-30
102 5.54739076640443e-34
107 1.65976833935966e-37
112 3.86926934152918e-41
117 7.11752228225861e-45
122 1.04492214399608e-48
127 1.2369628158818e-52
132 1.19179793462632e-56
137 9.42596636182606e-61
142 6.16771176250081e-65
147 3.36299438595356e-69
152 1.53823949151545e-73
157 5.93880361666429e-78
162 1.94645769286645e-82
167 5.44487163153814e-87
172 1.30647586200319e-91
177 2.70161803007538e-96
182 4.83580930639697e-101
187 7.52381681140081e-106
192 1.02147893266274e-110
197 1.21463836327617e-115
};
\addlegendentry{g=10}
\addplot [thick, forestgreen4416044]
table {%
2 0.00261300752960344
7 0.00224164013677549
12 0.00206819438579672
17 0.00199791713873856
22 0.00193014769205481
27 0.00186406976974627
32 0.00179290900804943
37 0.00168840825000655
42 0.00165429078121373
47 0.00167320117184585
52 0.00157372944262856
57 0.00151335847906368
62 0.00155556440678374
67 0.00153912956143622
72 0.00137427960064497
77 0.00146037717190552
82 0.000470737228340077
87 6.55723025264136e-05
92 4.91680577065505e-06
97 2.22976207118899e-07
102 6.58006455745173e-09
107 1.32977437497022e-10
112 1.9119065639309e-12
117 2.01492767382608e-14
122 1.59459923935019e-16
127 9.66854680607096e-19
132 4.56854031397487e-21
137 1.70712135235817e-23
142 5.109471172648e-26
147 1.23883245778314e-28
152 2.45775386522149e-31
157 4.02597343782387e-34
162 5.48973036578456e-37
167 6.27767009896649e-40
172 6.06117302415124e-43
177 4.97197524192304e-46
182 3.48504786361091e-49
187 2.09847879572844e-52
192 1.09084027491502e-55
197 4.91787160076439e-59
};
\addlegendentry{g=20}
\addplot [thick, crimson2143940]
table {%
2 0.00220711424636465
7 0.00191998298738163
12 0.00182751334079354
17 0.00171594345805887
22 0.0016340829720433
27 0.00159594905071389
32 0.00156573803985845
37 0.0015315248062587
42 0.00149227022279943
47 0.00144335030292992
52 0.00139273022028483
57 0.00138929486696936
62 0.00140292794913332
67 0.00131570427357359
72 0.00136187282657085
77 0.00126752355614861
82 0.00131693844933449
87 0.00131052857692016
92 0.0012496961456205
97 0.00121086901536945
102 0.00121515830379219
107 0.00128935619195853
112 0.00121854548485014
117 0.00113420484050598
122 0.000416680503543536
127 7.98861653785164e-05
132 9.33146560142724e-06
137 7.24324204544724e-07
142 3.95123453558221e-08
147 1.57599387193092e-09
152 4.73539873775146e-11
157 1.09733088007352e-12
162 1.99889215881516e-14
167 2.90819020376005e-16
172 3.42540559377219e-18
177 3.30467841935055e-20
182 2.63822344848898e-22
187 1.75866493581652e-24
192 9.86838438045789e-27
197 4.69516760191704e-29
};
\addlegendentry{g=30}
\addplot [thick, mediumpurple148103189]
table {%
2 0.00194317584252342
7 0.00174146205369618
12 0.00162015578897805
17 0.00154459311263063
22 0.00151288691912565
27 0.00145571962308896
32 0.00139924068492077
37 0.00136557327067527
42 0.00134402713830007
47 0.0013231150553622
52 0.00129890054092146
57 0.00127029675417688
62 0.0012397718499822
67 0.00122114881033169
72 0.00122939573853534
77 0.00122625155032581
82 0.00116965295415799
87 0.00118098132937105
92 0.00116414968791781
97 0.00114551766091404
102 0.00112474848978793
107 0.00115250822687336
112 0.00108877382894861
117 0.00106876026574586
122 0.0010901520498832
127 0.00110184898674521
132 0.00109470335993485
137 0.0010581662502208
142 0.000983740286595716
147 0.00103072338346228
152 0.00110953584564872
157 0.000948584666953878
162 0.00037926153728545
167 8.81667303364769e-05
172 1.34450885693646e-05
177 1.44278676398093e-06
182 1.14094302936847e-07
187 6.87203057386341e-09
192 3.23245444343551e-10
197 1.21114053500956e-11
};
\addlegendentry{g=40}
\end{axis}

\end{tikzpicture}
    \caption{}
\end{subfigure}
\begin{subfigure}[b]{0.45\textwidth}
    \centering
    \input{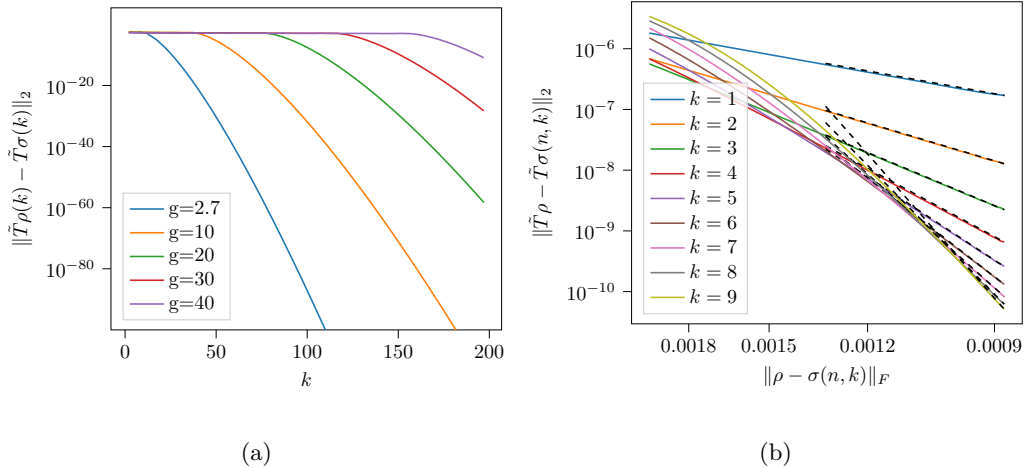}
    \caption{}
\end{subfigure}
\caption{(a) Difference of norms of results for different values of \(|g|\) plotted against the index of the largest non-zero off-diagonal (b) Convergence of norms for band-limited matrices for different limiting values k. The black dashed lines mark the rates from \Cref{the:bandlimited}.}
\label{fig:gen_band}
\end{figure}

For band-limited matrices we use restrictions of the operators \(\sigma(n,k,\epsilon), \rho(\epsilon)\) from \Cref{the:bandlimited} with \(\epsilon=10^{-15}\) and we fix \(g=2.7\) for all further examples. In \Cref{fig:gen_band}(b) we can see a loglog-plot of \(\|\tilde{T}_{N}\rho(\epsilon)-\tilde{T}_{N}\sigma(n,k,\epsilon)\|_{2}\) and \(\|\rho(\epsilon)-\sigma(n,k,\epsilon)\|_{F}\) for different \(k\). The values of \(n\) are scaled such that the norms of \(\|\rho(\epsilon)-\sigma(n,k,\epsilon)\|_{F}\) are in the same range for different \(k\). We can see that the norms converge with the rates predicted in \Cref{the:bandlimited}. 

For the case of exponential decay we use the example sequence from \Cref{the:exponential}. We observe in the loglog-plot in \Cref{fig:exppol}(a) that the rates are not close to the claimed ones and the convergence stops and becomes erratic after a certain point. This is caused by the fact that in this example the parameter \(n\) which controls the decay of the off-diagonal entries along the off-diagonal is exponentially increasing causing slower and slower decay. Therefore, the discretization error quickly dominates and prevents us from achieving the claimed rates in practice. From the shape of the curves, we can still see that as claimed the parameter \(b\) does not seem to influence the convergence further if it is already included in \(\Phi\). Additionally, the trend at the beginning suggest that the curves get steeper before later settling into linear behavior with the claimed rate which we cannot observe because of the truncation errors.
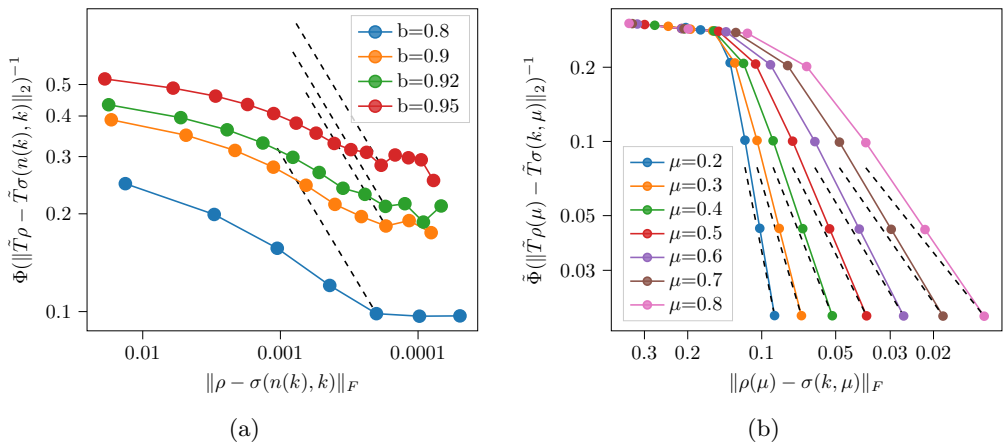
\begin{figure}
\begin{subfigure}[b]{0.45\textwidth}
    \centering
\begin{tikzpicture}[scale=0.75]

\definecolor{crimson2143940}{RGB}{214,39,40}
\definecolor{darkgray176}{RGB}{176,176,176}
\definecolor{darkorange25512714}{RGB}{255,127,14}
\definecolor{forestgreen4416044}{RGB}{44,160,44}
\definecolor{lightgray204}{RGB}{204,204,204}
\definecolor{steelblue31119180}{RGB}{31,119,180}

\begin{axis}[
legend cell align={left},
legend style={fill opacity=0.8, draw opacity=1, text opacity=1, draw=lightgray204},
log basis x={10},
log basis y={10},
tick align=outside,
tick pos=left,
x grid style={darkgray176},
xlabel={\(\displaystyle \|\rho-\sigma(n(k),k)\|_{F}\)},
xmin=39.3996636366369, xmax=27121.345085469,
xmode=log,
xtick style={color=black},
xtick={100,1000,10000},
xticklabels={0.01,0.001,0.0001},
y grid style={darkgray176},
ylabel={\(\displaystyle \Phi(\|\tilde{T}\rho-\tilde{T}\sigma(n(k),k)\|_{2})^{-1}\)},
ymin=0.0872631569165359, ymax=0.85542761110019,
ymode=log,
ytick style={color=black},
ytick={0.1,0.2,0.3,0.4,0.5},
yticklabels={
  \(\displaystyle {0.1}\),
  \(\displaystyle {0.2}\),
  \(\displaystyle {0.3}\),
  \(\displaystyle {0.4}\),
  \(\displaystyle {0.5}\)
}
]
\addplot [thick, steelblue31119180, mark=*, mark size=3, mark options={solid}]
table {%
74.7741544117887 0.247633687616076
330.458198262343 0.199109988123167
948.578873168034 0.15667167603036
2281.92616638037 0.120366058837139
4982.9516123463 0.0985509265576004
10234.7847368937 0.0968037966203434
20152.06081169 0.0970073859765302
};
\addlegendentry{b=0.8}
\addplot [thick, black, dashed, forget plot]
table {%
948.578873168034 0.318470922889066
4982.9516123463 0.0985509265576004
};
\addplot [thick, darkorange25512714, mark=*, mark size=3, mark options={solid}]
table {%
59.0808133464204 0.389447456953384
206.303426000251 0.349097892551186
467.905672324541 0.313413355514578
889.367789617831 0.278357352210897
1534.48058468785 0.244852123814813
2490.28039353812 0.213837390570758
3874.21231493338 0.196128090102872
5843.67862165455 0.183457953088952
8608.55178824831 0.190612253180355
12447.4769207546 0.174973513828982
};
\addlegendentry{b=0.9}
\addplot [thick, black, dashed, forget plot]
table {%
1534.48058468785 0.47224903219502
5843.67862165455 0.183457953088952
};
\addplot [thick, forestgreen4416044, mark=*, mark size=3, mark options={solid}]
table {%
56.5400033187382 0.433231514248599
188.940547074876 0.395428546706019
410.096820803825 0.362798679766576
745.965689932682 0.330330486542036
1231.70929281421 0.298472482305433
1912.9535861536 0.268068500487072
2848.05887958654 0.239690979677133
4111.13019422947 0.229590404105967
5795.81296172289 0.210803059650444
8020.01067769789 0.214812310014361
10931.7156266288 0.188625526639491
14716.1946137744 0.211423493532373
};
\addlegendentry{b=0.92}
\addplot [thick, black, dashed, forget plot]
table {%
1231.70929281421 0.630204174110202
5795.81296172289 0.210803059650444
};
\addplot [thick, crimson2143940, mark=*, mark size=3, mark options={solid}]
table {%
53.0254391213808 0.5202878613539
166.18126351578 0.487403833862322
338.276426977981 0.460593881274564
577.07549016141 0.433924555886658
893.614975326263 0.407001826125048
1301.59267765698 0.380288767026682
1817.38973504749 0.354258920097797
2460.30404602742 0.329272176707467
3252.89731622085 0.315179611908017
4221.42793584753 0.309180800263348
5396.36351628251 0.282094221403391
6812.97552666049 0.303441361457056
8512.02280623261 0.297508008590804
10540.5335344493 0.293244141940776
12952.697422869 0.253300631717826
};
\addlegendentry{b=0.95}
\addplot [thick, black, dashed, forget plot]
table {%
1301.59267765698 0.771119692246514
5396.36351628251 0.282094221403391
};
\end{axis}

\end{tikzpicture}
    \caption{}
\end{subfigure}
\begin{subfigure}[b]{0.45\textwidth}
    \centering
\begin{tikzpicture}[scale=0.75]

\definecolor{crimson2143940}{RGB}{214,39,40}
\definecolor{darkgray176}{RGB}{176,176,176}
\definecolor{darkorange25512714}{RGB}{255,127,14}
\definecolor{forestgreen4416044}{RGB}{44,160,44}
\definecolor{lightgray204}{RGB}{204,204,204}
\definecolor{mediumpurple148103189}{RGB}{148,103,189}
\definecolor{orchid227119194}{RGB}{227,119,194}
\definecolor{sienna1408675}{RGB}{140,86,75}
\definecolor{steelblue31119180}{RGB}{31,119,180}

\begin{axis}[
legend cell align={left},
legend style={fill opacity=0.8, draw opacity=1, text opacity=1, at={(0.03,0.03)},
  anchor=south west,draw=lightgray204},
log basis x={10},
log basis y={10},
tick align=outside,
tick pos=left,
x grid style={darkgray176},
xlabel={\(\displaystyle \|\rho(\mu)-\sigma(k,\mu)\|_{F}\)},
xmin=2.44361547132342, xmax=94.9448129507037,
xmode=log,
xtick style={color=black},
xtick={3.33333333333333,5,10,20,33.3333333333333,50},
xticklabels={0.3,0.2,0.1,0.05,0.03,0.02},
y grid style={darkgray176},
ylabel={\(\displaystyle \tilde{\Phi}(\|\tilde{T}\rho(\mu)-\tilde{T}\sigma(k,\mu)\|_{2})^{-1}\)},
ymin=0.0170480792949748, ymax=0.345356646293983,
ymode=log,
ytick style={color=black},
ytick={0.03,0.05,0.1,0.2},
yticklabels={
  \(\displaystyle {0.03}\),
  \(\displaystyle {0.05}\),
  \(\displaystyle {0.1}\),
  \(\displaystyle {0.2}\)
}
]
\addplot [thick, steelblue31119180, mark=*, mark size=2, mark options={solid}]
table {%
4.90906501596675 0.289746350721548
5.6390349084145 0.283482211029827
6.47755012306659 0.280366542050913
7.44075117077743 0.208565369986762
8.54717862981429 0.101154707657402
9.81813003193349 0.0444140584988837
11.278069816829 0.0196560298265912
};
\addlegendentry{$\mu$=0.2}
\addplot [thick, black, dashed, forget plot]
table {%
8.54717862981429 0.0786241193063651
11.278069816829 0.0196560298265912
};
\addplot [thick, darkorange25512714, mark=*, mark size=2, mark options={solid}]
table {%
4.16603035387814 0.293207625571703
5.12898499600242 0.285394206647152
6.31452122395828 0.280867626832964
7.77408752782414 0.2081073326328
9.57102442873508 0.100921515746646
11.7833132554249 0.0443534077734114
14.5069602851095 0.0196418978552897
};
\addlegendentry{$\mu$=0.3}
\addplot [thick, black, dashed, forget plot]
table {%
9.57102442873508 0.0785675914211591
14.5069602851095 0.0196418978552897
};
\addplot [thick, forestgreen4416044, mark=*, mark size=2, mark options={solid}]
table {%
3.67073806885336 0.295927096205127
4.84356792022723 0.286560619528667
6.39112618710564 0.280630369137886
8.4331415626337 0.207262736319054
11.1275970045629 0.10061279246569
14.6829522754134 0.0442805168366846
19.3742716809089 0.0196256854077487
};
\addlegendentry{$\mu$=0.4}
\addplot [thick, black, dashed, forget plot]
table {%
11.1275970045629 0.0785027416309946
19.3742716809089 0.0196256854077487
};
\addplot [thick, crimson2143940, mark=*, mark size=2, mark options={solid}]
table {%
3.33876938854894 0.297992266178173
4.72173295092203 0.287082448241945
6.67753877709787 0.27977104376737
9.44346590184406 0.206100843848503
13.3550775541957 0.100241774885417
18.8869318036881 0.0441974189462015
26.7101551083915 0.019607726021562
};
\addlegendentry{$\mu$=0.5}
\addplot [thick, black, dashed, forget plot]
table {%
13.3550775541957 0.0784309040862482
26.7101551083915 0.019607726021562
};
\addplot [thick, mediumpurple148103189, mark=*, mark size=2, mark options={solid}]
table {%
3.11752314500823 0.299501199790166
4.72528147736857 0.287067035075227
7.16218741667227 0.278408086283185
10.8558461199025 0.20468936789026
16.4543858074238 0.0998211629801074
24.9401851600658 0.0441060624803403
37.8022518189486 0.019588339803692
};
\addlegendentry{$\mu$=0.6}
\addplot [thick, black, dashed, forget plot]
table {%
16.4543858074238 0.078353359214768
37.8022518189486 0.019588339803692
};
\addplot [thick, sienna1408675, mark=*, mark size=2, mark options={solid}]
table {%
2.97370300251679 0.300548016561377
4.83079477969199 0.286614626870679
7.84764927222002 0.276649370952441
12.748543854248 0.203088280905745
20.7100705913309 0.0993621054510494
33.6436089330307 0.0440081588767137
54.6542039558525 0.0195678083009752
};
\addlegendentry{$\mu$=0.7}
\addplot [thick, black, dashed, forget plot]
table {%
20.7100705913309 0.0782712332039008
54.6542039558525 0.0195678083009752
};
\addplot [thick, orchid227119194, mark=*, mark size=2, mark options={solid}]
table {%
2.88589193898863 0.301215458597635
5.02462970619658 0.285811954785253
8.74838844216775 0.274586615102749
15.2318289725249 0.201347445558334
26.5201545841235 0.0988738448443815
46.174271023818 0.0439051296291866
80.3940752991453 0.0195463656429771
};
\addlegendentry{$\mu$=0.8}
\addplot [thick, black, dashed, forget plot]
table {%
26.5201545841235 0.0781854625719084
80.3940752991453 0.0195463656429771
};
\end{axis}

\end{tikzpicture}
    \caption{}
\end{subfigure}
\caption{(a) Convergence of norms of matrices with exponential decay. This case is severely corrupted by truncation errors due to the nature of the required sequence. The black dashed lines mark the rate from \Cref{the:exponential} and they roughly mark the point where large truncation errors begin to appear. (b) Convergence of norms for matrices with polynomial decay with different decay rates \(\mu\). The black dashed lines mark the rates from \Cref{the:polynomial}.}
\label{fig:exppol}
\end{figure}

Finally, we consider the case of polynomial decay. We use the example from \Cref{the:polynomial} for different \(\mu\). Due to the exponential increase of the required matrix size and the very small size of the norms in the image space we are again limited in the range of examples we can numerically investigate. The results shown in \Cref{fig:exppol}(b) still agree quite well with the predicted rates. This difference to the results of exponential decay can be explained by the fact, that in the polynomial case the nonzero entries are limited to a finite submatrix and therefore the truncation errors are smaller.  

Apart from the case of exponential decay, for which an important part of the example sequence is numerically unattainable, these numerical results support our theoretical findings.

\subsection{Numerical stability of the inverse problem}
Now we consider the stability of the inverse problem. While the discrete problem is always theoretically well-posed because of the restriction to a finite dimensional subspace, we will show that the additional restriction to the set of positive semi-definite matrices significantly increases the stability. To achieve this, we first consider the example of a (simulated) PINEM state. We add different amounts of white noise and solve the problems
\begin{align}
    \underset{\rho\in\BopIpsdtr}{\argmin}&\|T_{N,\Theta}\rho-\yobs\|_{2}^{2}\label{prob:const}\\
    \underset{\rho\in\Herm(N)}{\argmin}&\|T_{N,\Theta}\rho-\yobs\|_{2}^{2}\label{prob:unconst}
\end{align}
with and without the constraint to the set of valid density matrices. Usually the noise is modeled as Poisson noise, but we consider Gaussian white noise here instead to see the effect of non-zero entries corresponding to off-diagonals which are far away from the main diagonal. We orient ourselves on the supplementary material of \cite{Priebe:17} and choose \(N=41\) and \(M=100\) different equally spaced angles \(\theta\). We simulate a PINEM state with a pump frequency of \(g_{\text{pump}}=1.73\) and include a slight phase jitter to get a state that is not pure. We then solve the problems  (\ref{prob:const}) and (\ref{prob:unconst}) for different noise levels and different values of \(|g|\). We solve the unconstrained system using the CG method and we use projected gradient descent (see \cite{Bolduc:17}) for the constrained problem. In both cases we stop the iteration if the residual is smaller than \(1.1\) times the noise level. The results in \Cref{fig:conv}(a) show that for larger \(|g|\) the constraint greatly improves the convergence behavior while for smaller \(|g|\) it does not make a huge difference. This can be explained by the observation that the positive semi-definiteness constraint is especially useful for suppressing noise in high order off-diagonals, but has little influence on the closer off-diagonals. If \(|g|\) is chosen too small the low order off-diagonals already do not contribute much to the measurement and cannot be reconstructed well. We also compute the condition numbers of \(T_{N,\Theta}^{*}T_{N,\Theta}\). They are between \(5.72\cdot 10^{27}\) for \(g=6.92\) and \(1.62\cdot 10^{70}\) for \(g=0.865\) showing that as expected the linear system is severely ill-conditioned. The CG-solver which is used to numerically obtain reasonable results therefore also regularizes the problem by early stopping \cite[Chapter~7]{Engl:96}. This explains the better than expected convergence behavior of the unconstrained problem. 

\begin{figure}
\begin{subfigure}[b]{0.45\textwidth}
    \centering
    \input{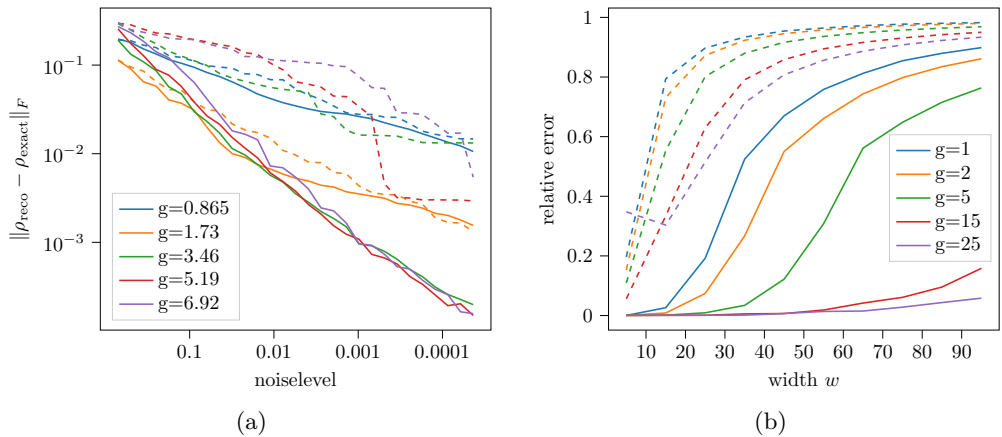}
    \caption{}
\end{subfigure}
\begin{subfigure}[b]{0.45\textwidth}
    \centering
\begin{tikzpicture}[scale=0.75]

\definecolor{crimson2143940}{RGB}{214,39,40}
\definecolor{darkgray176}{RGB}{176,176,176}
\definecolor{darkorange25512714}{RGB}{255,127,14}
\definecolor{forestgreen4416044}{RGB}{44,160,44}
\definecolor{lightgray204}{RGB}{204,204,204}
\definecolor{mediumpurple148103189}{RGB}{148,103,189}
\definecolor{steelblue31119180}{RGB}{31,119,180}

\begin{axis}[
legend cell align={left},
legend style={
  fill opacity=0.8,
  draw opacity=1,
  text opacity=1,
  at={(0.98,0.41)},
  anchor=east,
  draw=lightgray204
},
tick align=outside,
tick pos=left,
x grid style={darkgray176},
xlabel={width \(w\)},
xmin=0.5, xmax=99.5,
xtick style={color=black},
xtick={10,20,30,40,50,60,70,80,90},
xticklabels={
  \(\displaystyle 10\),
  \(\displaystyle 20\),
  \(\displaystyle 30\),
  \(\displaystyle 40\),
  \(\displaystyle 50\),
  \(\displaystyle 60\),
  \(\displaystyle 70\),
  \(\displaystyle 80\),
  \(\displaystyle 90\)
},
y grid style={darkgray176},
ylabel={relative error},
ymin=-0.0488085031062922, ymax=1.0314116568501,
ytick style={color=black}
]
\addplot [thick, steelblue31119180]
table {%
5 0.000350662484309059
15 0.0262234490318801
25 0.192149114662186
35 0.524769381972554
45 0.669819731047951
55 0.758340077027453
65 0.812232313708964
75 0.854572538567842
85 0.879461186091604
95 0.898578993302656
};
\addlegendentry{g=1}
\addplot [thick, steelblue31119180, dashed, forget plot]
table {%
5 0.19614152766659
15 0.794128652877379
25 0.897188635820757
35 0.934080652517153
45 0.95421780550495
55 0.964251535529903
65 0.972288878827961
75 0.977216261132066
85 0.980197771709386
95 0.982310740488443
};
\addplot [thick, darkorange25512714]
table {%
5 0.000420374962801603
15 0.00855431437685747
25 0.0740298254722191
35 0.267160735606004
45 0.550340352985785
55 0.660624687155995
65 0.74334502371595
75 0.797842655965207
85 0.834711177063922
95 0.861116703927532
};
\addlegendentry{g=2}
\addplot [thick, darkorange25512714, dashed, forget plot]
table {%
5 0.15296647269215
15 0.731408942516316
25 0.872385009439072
35 0.92327441877717
45 0.945523408986183
55 0.958193138723864
65 0.966353717694948
75 0.972343395408373
85 0.976457676205902
95 0.979893838615035
};
\addplot [thick, forestgreen4416044]
table {%
5 0.000292413255361826
15 0.00123816937826743
25 0.0092038907941443
35 0.0341969198553126
45 0.121942895563856
55 0.307807298047426
65 0.561331505519207
75 0.64813547112259
85 0.71513027555546
95 0.763629023446055
};
\addlegendentry{g=5}
\addplot [thick, forestgreen4416044, dashed, forget plot]
table {%
5 0.109305318222973
15 0.553576386110613
25 0.801511367418451
35 0.879807650580109
45 0.916462622382311
55 0.936030217083638
65 0.948822115387131
75 0.957609828112523
85 0.963867633800982
95 0.968609163316451
};
\addplot [thick, crimson2143940]
table {%
5 0.0004830103498261
15 0.000629168285753984
25 0.000936603358700611
35 0.00555764648526923
45 0.0066686856310395
55 0.0184810657500403
65 0.0417752781651995
75 0.0609908209778896
85 0.0955192635858573
95 0.158376623633589
};
\addlegendentry{g=15}
\addplot [thick, crimson2143940, dashed, forget plot]
table {%
5 0.0560266667521243
15 0.334087613235171
25 0.62986940784544
35 0.790808543035562
45 0.858636512401064
55 0.894418670061107
65 0.916108512194674
75 0.930896161409466
85 0.942097117345372
95 0.950092456184298
};
\addplot [thick, mediumpurple148103189]
table {%
5 0.000720736156575445
15 0.000768364860999673
25 0.000912461596653899
35 0.00125128320850029
45 0.00682050062611694
55 0.0136851655663004
65 0.0150821695804476
75 0.0280512392897257
85 0.0434565488645184
95 0.0583347000618215
};
\addlegendentry{g=25}
\addplot [thick, mediumpurple148103189, dashed, forget plot]
table {%
5 0.347785618317889
15 0.303377815341775
25 0.51480144006489
35 0.713550530411464
45 0.807545913557193
55 0.856116141841629
65 0.886680943997495
75 0.907852201513282
85 0.922718168521445
95 0.933704801067144
};
\end{axis}

\end{tikzpicture}
    \caption{}
\end{subfigure}
\caption{(a) Comparison of reconstruction errors with restriction (regular) and without restriction (dashed) for different noise levels. (b) Comparison of relative reconstruction errors for matrices \(\rho(w)\) with different width \(w\) from exact data. The reconstructions without restriction are again marked by dashed lines.}
\label{fig:conv}
\end{figure}

Finally, we take a closer look at the stability for larger matrices and the relationship between \(g\) and the support of the reconstructed density matrix. For this we consider matrices with size \(N=201\) and measurements with \(500\) different equally spaced angles. To simplify the experiments we consider matrices of the following form
\begin{align*}
    \rho(w)_{j,k}=\begin{cases}
        \frac{1}{w} \quad\text{ if } \lfloor\frac{N-w}{2}\rfloor\leq j,k<\lfloor\frac{N-w}{2}\rfloor+w\\
        0 \quad\text{ else.}
    \end{cases}    
\end{align*}
They consists of a constant centered square with width \(w\). We then reconstruct the density matrices from exact simulated data with and without constraint with the same algorithms as before. In \Cref{fig:conv}(b) we can see the relative error \(\frac{\|\rho_{\text{reco}}-\rho(w)\|_{F}}{\|\rho(w)\|_{F}}\) in relation to the width \(w\). The plot clearly shows the stabilizing effect of the constraint and the importance of choosing \(|g|\) sufficiently large to reconstruct matrices with larger width. In the unconstrained case however the effect of choosing a larger \(|g|\) is much smaller due to the instability.

All in all our findings suggest, that for practical applications with finite matrices and sufficiently large \(|g|\) the constraint to the set of density matrices is enough to achieve a stable reconstruction and no further regularization is necessary. This observation agrees with the more recent findings from the experimental physics community \cite{Jeng:25} which prefer reconstructions without additional regularization.
\section{Conclusions}\label{sec:conclusion}
We have shown the stability of the constrained PINEM quantum tomography problem which demonstrates the regularizing effect of the positive semi-definiteness and trace constraint. This analysis does not result in any general stability estimates and we have shown that such estimates do not exists as the inverse is continuous but not uniformly continuous on the non-compact restricted domain. We furthermore investigated the optimality of the rates shown under the assumption of additional decay conditions in \cite{Shi:20} and found that the Hölder exponent in the case of band-limited matrices cannot be improved. For the cases of polynomially and exponentially decaying off-diagonals we got lower bounds which differ from the previously derived rates only by a relatively small constant. All in all, this answers the open theoretical questions posed in \cite{Shi:20}.

We also showed that the decomposition of the forward operator into Fourier transforms and multiplication operators can be done in the discrete case. This allowed us to show that \(2N-1\) angles are enough to uniquely determine a density matrix of size \(N\times N\). The derived decomposition is also relevant for numerical computations as it can be computed much faster than previous methods by exploiting the fast Fourier transform. We additionally computed a lower bound on the norm of the inverse of the discrete forward operator which provides a more detailed understanding on how the instability relates to the experimental parameters. 

Our numerical experiments suggest, that for suitably chosen experimental parameters and truncation of the domain regularization in the form of penalty terms as it is done in the SQUIRRELS algorithm \cite{Priebe:17} is not necessary and the restriction to the set of density matrices is already enough to regularize the problem. This agrees with our theoretical stability result. However, we also observed that if the experimental parameter \(|g|\) is chosen too small, off-diagonals which are further away cannot be recovered well even with the additional constraint and without the presence of noise. Further research could investigate whether in this case additional apriori information or different reconstruction algorithms could improve the recovery. 
\section*{Data Availability Statement}
The code associated with this article is available in `GRO.data', under the reference

\noindent https://doi.org/10.25625/IPG2F4.
\appendix
\section{Proof of \Cref{lem:seq_prop}}\label{sec:proof_seq}
\begin{proof}
The proof consists of four parts. First we show that the matrices are valid density matrices. Because we normalize with the trace, we just have to show that they are positive semi-definite. We do this by applying \Cref{lem:gershgorin} and then it remains to prove that \(\tau(\epsilon)_{l,l}\geq\eta(n,k,\epsilon)_{l,l+k}+\eta(n,k,\epsilon)_{l,l-k}\). 
For this, we first show the bound \(\tau(\epsilon)_{l,l}\geq n^{-\epsilon}\hat{h}^{(n)}(l)\). It is enough to consider the case \(l>0\) and compute
\begin{align*}
    \frac{n^{-\epsilon}\hat{h}^{(n)}(l)}{\tau(\epsilon)_{l,l}}=\frac{1}{2}n^{1-\epsilon}\frac{1-\cos(\frac{l}{n})}{l^{2}}l^{1+\epsilon}=\frac{1}{2}\frac{1-\cos(\frac{l}{n})}{\left(\frac{l}{n}\right)^{1-\epsilon}}\leq\frac{1}{2} \sup_{t\in(0,\infty)}\frac{1-\cos(t)}{t^{1-\epsilon}}\leq 1.
\end{align*}
Note that this bound holds even though \(\tau(\epsilon)\) does not depend on \(n\).

With this we can compute
    \begin{align*}
        \eta(n,k,\epsilon)_{l,l+k}+\eta(n,k,\epsilon)_{l,l-k}&\leq\frac{1}{2}(2|k|)^{-1-\epsilon}(\tau(\epsilon)_{l,l}+\tau(\epsilon)_{l-k,l-k})\\
        &\leq(2|k|)^{-1-\epsilon}\max(\tau(\epsilon)_{l,l},\tau(\epsilon)_{l-k,l-k}).
    \end{align*}
    We then only consider the case where the maximum is attained for the second value and get
    \begin{align*}
        (2|k|)^{-1-\epsilon}\tau(\epsilon)_{l-k,l-k}&=\begin{cases}
            (2|l|)^{-1-\epsilon}\tau(\epsilon)_{0,0}\quad l=k\\
            (2|k|)^{-1-\epsilon}(|l-k|)^{-1-\epsilon}\tau(\epsilon)_{0,0}\quad l\neq k\\
        \end{cases}\\
        &\leq \tau(\epsilon)_{l,l}.
    \end{align*}
    This finally shows \(\eta(n,k,\epsilon)_{l,l+k}+\eta(n,k,\epsilon)_{l,l-k}\leq\tau(\epsilon)_{l,l}\) and finishes the proof of the positive semi-definiteness.

    The Hilbert-Schmidt norm can be computed using Parseval's identity
    \begin{align*}
    \|\eta(n,k,\epsilon)\|_{2}^{2}&=\frac{1}{2}\sum_{l}\left((2|k|)^{-1-\epsilon}n^{-\epsilon}\hat{h}^{(n)}(l)\right)^{2}=\frac{1}{2}(2|k|)^{-2-2\epsilon}n^{-2\epsilon}\sum_{l}\left(\hat{h}^{(n)}(l)\right)^{2}\\
    &=\frac{1}{2}(2|k|)^{-2-2\epsilon}n^{-2\epsilon}\int_{-\pi}^{\pi}h^{(n)}(x)^{2}\dd x=(2|k|)^{-2-2\epsilon}n^{-2\epsilon}\frac{1}{3n}\\
    \|\eta(n,k,\epsilon)\|_{2}&=\frac{1}{\sqrt{3}}(2|k|)^{-1-\epsilon}n^{-\frac{1}{2}-\epsilon}.
\end{align*}
We can also compute the sum of the absolute values of the non-zero off-diagonal where we use that all values are non-negative and we can ignore the absolute values
\begin{align*}
    \sum_{l\in\setZ}|\eta(n,k,\epsilon)_{l,l+k}|&=n^{-\epsilon}\frac{1}{2}(2|k|)^{-1-\epsilon}\sum_{l\in\setZ}|\hat{h}^{(n)}(l)|=n^{-\epsilon}\frac{1}{2}(2|k|)^{-1-\epsilon}\sum_{l\in\setZ}\hat{h}^{(n)}(l)\\
    &=n^{-\epsilon}\frac{1}{2}(2|k|)^{-1-\epsilon}h^{(n)}(0)=n^{-\epsilon}\frac{1}{2}(2|k|)^{-1-\epsilon}.
\end{align*}
Note that this also holds for the case \(k<0\) as this just translates the Fourier coefficients. 

We can also bound \(\|T\eta(n,k,\epsilon)\|_{2}\) as follows using the decomposition from \Cref{the:decomp_cont}
\begin{align*}
    \|T\eta(n,k,\epsilon)\|_{2}^{2}&=\frac{1}{2}(2|k|)^{-2-2\epsilon}n^{-2\epsilon}2\pi\int_{-\pi}^{\pi}J_{k}(4|g|\sin(\frac{x}{2}))^{2}h^{(n)}(x)^{2}\dd x.
\end{align*}
Using the upper bound for the Bessel function \cite{Abramowitz:65}  we obtain
\begin{align*}
    \|T\eta(n,k,\epsilon)\|_{2}^{2}&\leq\frac{1}{2}(2|k|)^{-2-2\epsilon}n^{-2\epsilon}2\pi\int_{-\frac{1}{n}}^{\frac{1}{n}}
    \frac{|4|g|\sin(\frac{x}{2})|^{2k}}{2^{2k}(k!)^2}h^{(n)}(x)^{2}\dd x\\
     &\leq\frac{1}{2}(2|k|)^{-2-2\epsilon}n^{-2\epsilon}2\pi\frac{|g|^{2k}}{(k!)^{2}}\int_{-\frac{1}{n}}^{\frac{1}{n}}x^{2k}h^{(n)}(x)^{2}\dd x\\
     &\leq2(2|k|)^{-2-2\epsilon}n^{-2\epsilon}2\pi\frac{|g|^{2k}}{(k!)^{2}}\frac{1}{(2k+1)(2k+2)(2k+3)}n^{-2k-1}\\
     \|T\eta(n,k,\epsilon)\|_{2}&\leq2\sqrt{\pi}\frac{(2k)^{-1-\epsilon}|g|^{k}}{k!\sqrt{(2k+1)(2k+2)(2k+3)}}n^{-k-\frac{1}{2}-\epsilon}.
\end{align*}
\end{proof}

\bibliographystyle{abbrv}
\bibliography{references}

\end{document}